\documentclass[leqno,11pt]{amsart}

\usepackage{amsmath,amstext,amssymb,amsopn,amsthm,mathrsfs}
\usepackage{verbatim}

\title[Potential spaces for Jacobi expansions]
{On potential spaces \\ related to Jacobi expansions}

\author[B{.} Langowski]{Bartosz Langowski}
\address{Bartosz Langowski \newline
			Institute of Mathematics and Computer Science \newline
      Wroc\l{}aw University of Technology       \newline
      Wyb{.} Wyspia\'nskiego 27,
      50--370 Wroc\l{}aw, Poland      
      }
\email{bartosz.langowski@pwr.wroc.pl}

\allowdisplaybreaks

\theoremstyle{plain}
\newtheorem{thm}{Theorem}[section]

\newtheorem{lem}[thm]{Lemma}
\newtheorem{prop}[thm]{Proposition}

\newtheorem{remark}[thm]{Remark}

\theoremstyle{definition}

\theoremstyle{remark}
\newtheorem*{rem*}{Remark}

\theoremstyle{plain}

\DeclareMathOperator{\spann}{span}

\DeclareMathOperator{\support}{supp}

\DeclareMathOperator{\id}{Id}

\def\N{\mathbb N}
\def\Z{\mathbb Z}

\def\e{e}

\def\P{\mathcal P}

\def\m{\mu} 						
\def\ab{\alpha,\beta}
\def\J{\mathcal J} 			
\def\q{\mathfrak q}

\def\ph{\phi_n^{\ab}}

\def\a{\alpha}
\def\b{\beta}
\def\ab{\alpha,\beta}
\def\t{\theta}
\def\vp{\varphi}
\def\st{\sin \frac{\theta}{2}}
\def\svp{\sin \frac{\varphi}{2}}
\def\ct{\cos \frac{\theta}{2}}
\def\cvp{\cos \frac{\varphi}{2}}

\def\pia{d\Pi_{\alpha}(u)}
\def\pib{d\Pi_{\beta}(v)}

\def\q{\mathfrak{q}}

\def\piK{d\Pi_{\alpha, K}(u)}
\def\piR{d\Pi_{\beta, R}(v)}

\def\P{\mathcal P}

\def\m{\mu}

\begin{document}

\begin{abstract}
We investigate potential spaces associated with Jacobi expansions.
We prove structural and Sobolev-type embedding theorems for these spaces.
We also establish their characterizations in terms of suitably defined fractional square functions.
Finally, we present sample applications of the Jacobi potential spaces connected with a PDE problem.
\end{abstract}

\maketitle

\footnotetext{

Research supported by the National Science Centre of Poland, project no.\ 2013/09/N/ST1/04120.
}

\section{Introduction} \label{sec:intro}

This paper is a continuation of our study from \cite{L2}, where Sobolev spaces and potential spaces
in the context of expansions into Jacobi trigonometric `functions' were investigated. The main
achievement of \cite{L2} is a proper definition of Jacobi Sobolev spaces in terms of suitably
chosen higher-order distributional derivatives, so that these spaces coincide with the Jacobi potential
spaces with certain parameters (see Section \ref{sec:prel} for details). The latter spaces 
are defined similarly
as in the classical situation, via integral operators arising from negative powers of the Jacobi
Laplacian (or its shift, in some cases). 

In the present paper we focus on the Jacobi potential spaces. Nevertheless, in view of what was just 
said above, our results implicitly pertain also to the Jacobi Sobolev spaces. We prove structural and
Sobolev-type embedding theorems for the potential spaces (Theorems \ref{thm:struc} and \ref{thm:embed}). 
We also establish their
characterizations in terms of suitably defined fractional square functions 
(Theorems \ref{equivfun} and \ref{equivfun'}). This part is motivated by
the recent results of Betancor et.\ al.\ \cite{betsq}, and the associated analysis uses the theory of
vector-valued Calder\'on-Zygmund operators on spaces of homogeneous type. As a result of independent
interest, we prove $L^p$-boundedness of the `vertical' fractional $g$-functions associated with Jacobi
trigonometric `function' and polynomial expansions (Theorems \ref{ogrgff} and \ref{ogrgf}).
Finally, inspired by some of the results in \cite{BR,BT1,BT2}, we present sample applications of the Jacobi
potential spaces connected with a Cauchy PDE problem based on the Jacobi Laplacian.

We believe that our results enrich the line of research concerning Sobolev and potential spaces related
to classical discrete and continuous orthogonal expansions, see in particular 
\cite{betancor,betsq,BT1,BT2,Graczyk,L2,RT}; see also \cite{BU1,BU2} where some results on Jacobi potential
spaces can be found, though in a different Jacobi setting and with a different approach from ours.
We point out that intimately connected to potential spaces are potential operators, and in the above-mentioned
contexts they were studied intensively and thoroughly in the recent past. We refer the interested readers
to \cite{Nowak&Roncal,NoSt1,NoSt2,NoSt3} and also to references given in these works. 
In particular, \cite{Nowak&Roncal} delivers a solid ground for our developments. 

The paper is organized as follows. In Section \ref{sec:prel} we introduce the Jacobi setting and
basic notions. In Section \ref{sec:structural} we prove the 
structural and embedding theorems announced above.
Section \ref{sec:char} contains the fractional square function characterizations of the Jacobi
potential spaces. Section~\ref{sec:convergence} is devoted to sample applications of the potential
spaces. Finally, in Section \ref{sec:g-functions}
we prove the $L^p$ results for the fractional square functions needed in Section \ref{sec:char}.

\textbf{Notation.} Throughout the paper we use a standard notation
We write $X\lesssim Y$ to indicate that $X\leq CY$ 
with a positive constant $C$ independent of significant quantities. We shall write $X \simeq Y$ when
simultaneously $X \lesssim Y$ and $Y \lesssim X$.

\textbf{Acknowledgment.}
The author would like to express his gratitude to Professor Adam Nowak for 
his constant support during the preparation of this paper.

\section{Preliminaries} \label{sec:prel}

Given parameters $\ab > -1$, the \emph{Jacobi trigonometric functions} are defined as
\begin{equation*}
\phi_n^{\ab}(\theta):=\Psi^{\ab}(\theta)\,\P_n^{\ab}(\theta),\qquad\theta\in(0,\pi), \quad n \ge 0,
\end{equation*}
where
$$
\Psi^{\ab}(\theta) := \Big( \sin\frac{\theta}2\Big)^{\alpha+1/2} \Big(\cos\frac{\theta}2\Big)^{\beta+1/2}
$$
and
$$
\P_n^{\ab}(\theta) := c_n^{\ab}\, P_n^{\ab}(\cos\theta)
$$
with $P_n^{\ab}$ denoting the classical \emph{Jacobi polynomials} as defined in Szeg\"o's monograph \cite{Sz}
and $c_n^{\ab}$ being normalizing constants. The system $\{\phi_n^{\ab} : n \ge 0\}$ is an orthonormal
basis in $L^2(0,\pi)$. This basis consists of eigenfunctions of the \emph{Jacobi Laplacian}
\begin{equation*} 
L_{\ab}=-\frac{d^2}{d\theta^2} -\frac{1-4\alpha^2}{16\sin^2\frac{\theta}{2}}-\frac{1-4\beta^2}{16\cos^2\frac{\theta}{2}}
 = D_{\ab}^* D_{\ab} + A^2_{\ab}; 
\end{equation*}
here $A_{\ab}=(\alpha+\beta+1)/2$,
$D_{\ab}=\frac{d}{d\theta}-\frac{2\alpha+1}{4}\cot\frac{\theta}{2}+\frac{2\beta+1}{4}\tan\frac{\theta}{2}$
is the first order `derivative' naturally associated with $L_{\ab}$, and 
$D_{\ab}^* = D_{\ab}-2\frac{d}{d\theta}$ is its formal adjoint in $L^2(0,\pi)$.
The eigenvalue corresponding to $\phi_n^{\ab}$ is
$$
\lambda_n^{\ab} := \big(n+ A_{\ab}\big)^2.
$$
It is well known that $L_{\ab}$, considered initially on $C_c^2(0,\pi)$, has a non-negative self-adjoint
extension to $L^2(0,\pi)$ whose spectral resolution is discrete and given by the $\phi_n^{\ab}$.
We denote this extension by still the same symbol $L_{\ab}$. Notice that for some choices of $\alpha$
and $\beta$ we get the same differential operator $L_{\ab}$, nevertheless the resulting self-adjoint
extensions are different. Some problems in harmonic analysis related to $L_{\ab}$ were investigated
recently in \cite{L2,Nowak&Roncal,NS2,Stempak}.

When $\ab \ge -1/2$, the functions $\phi_n^{\ab}$ belong to all $L^p(0,\pi)$, $1 < p < \infty$.
However, if $\alpha< -1/2$ or $\beta < -1/2$, then $\phi_n^{\ab}$ are in $L^p(0,\pi)$ if and only if
$p < -1/\min(\alpha+1/2,\beta+1/2)$. This leads to the so-called \emph{pencil phenomenon}
manifesting in the restriction $p \in E(\ab)$ for $L^p$ mapping properties of various harmonic
analysis operators associated with $L_{\ab}$. Here
$$
E(\ab) := \big( p'(\ab), p(\ab) \big)
$$
with
$$
p(\alpha,\beta) := \begin{cases}
											\infty, & \ab \ge -1/2, \\
											-1/\min(\alpha+1/2,\beta+1/2), & \textrm{otherwise}
										\end{cases}
$$
and $p'$ denoting the conjugate exponent of $p$, $1/p+1/p'=1$.
Recall that (see \cite[Lemma 2.3]{Stempak}) the subspace
$$
S_{\ab} := \spann\{\ph: n \ge 0 \}
$$
is dense in $L^p(0,\pi)$ provided that $1 \le p < p(\ab)$.

We denote by $\{H_t^{\ab}\}_{t \ge 0}$ the \emph{Poisson-Jacobi semigroup}, that is the semigroup of
operators generated in $L^2(0,\pi)$ by the square root of $L_{\ab}$. In view of the spectral theorem,
for $f \in L^2(0,\pi)$ and $t\ge 0$ we have
\begin{equation} \label{JPser}
H_t^{\ab}f = \sum_{n=0}^{\infty} \exp\Big(-t\sqrt{\lambda_{n}^{\ab}}\Big)\, a_n^{\ab}(f)\, \phi_n^{\ab},
\end{equation}
where
$$
a_n^{\ab}(f) := \int_0^{\pi} f(\theta) \phi_n^{\ab}(\theta)\, d\theta
$$
is the $n$th Fourier-Jacobi coefficient of $f$. The series in \eqref{JPser} converges in $L^2(0,\pi)$.
Moreover, if $t>0$, it converges pointwise and that even for $f\in L^p(0,\pi)$, $p > p'(\ab)$, 
defining a smooth
function both in $t$ and the space variable. Thus \eqref{JPser} provides an extension of $\{H_t^{\ab}\}_{t>0}$
to the above $L^p$ spaces (which we denote by still the same symbol). The pointwise convergence and
smoothness are easily seen with the aid of the polynomial bound (cf.\ \cite[(7.32.2)]{Sz})
\begin{equation} \label{bphi}
|\phi_{n}^{\ab}(\theta)| \le C \, \Psi^{\ab}(\theta)\, 
	(n+1)^{1/2+\max\{\alpha,\beta,-1/2\}}, \qquad \theta \in (0,\pi),
 \quad n \ge 0,
\end{equation}
and the resulting polynomial growth in $n$ of $a_n^{\ab}(f)$. Furthermore, $\{H_t^{\ab}\}_{t>0}$
has an integral representation 
$$
H_t^{\ab}f(\theta) = \int_0^{\pi} H_t^{\ab}(\theta,\varphi)f(\varphi)\, d\varphi, \qquad t >0, \quad 
	\theta \in (0,\pi),
$$
valid for $f \in L^p(0,\pi)$, $p > p'(\ab)$. We note that sharp estimates of the Poisson-Jacobi
kernel $H_t^{\ab}(\theta,\varphi)$ follow readily from \cite[Theorem A.1 in the appendix]{NS2} and 
\cite[Theorem 6.1]{parameters}.

Next, we gather some facts about potential operators associated with $L_{\ab}$. Let $\sigma > 0$.
We consider the \emph{Riesz type potentials} $L_{\ab}^{-\sigma}$ assuming that $\alpha+\beta\neq -1$
(when $\alpha+\beta = -1$, the bottom eigenvalue of $L_{\ab}$ is $0$) and the \emph{Bessel type potentials}
$(\id+L_{\ab})^{-\sigma}$ with no restrictions on $\alpha$ and $\beta$. Clearly, these operators are
well defined spectrally and bounded in $L^2(0,\pi)$. Moreover, both $L_{\ab}^{-\sigma}$ and 
$(\id+L_{\ab})^{-\sigma}$ possess integral representations that extend actions of these potentials to
$L^p(0,\pi)$, $p > p'(\ab)$, see \cite{Nowak&Roncal}. We keep the same notation for the corresponding
extensions. According to \cite[Proposition 2.4]{L2}, $L_{\ab}^{-\sigma}$ and 
$(\id+L_{\ab})^{-\sigma}$ are bounded and one-to-one on $L^p(0,\pi)$ for $p \in E(\ab)$. 
An exhaustive study of $L^p-L^q$ mapping properties of the potential operators is contained in 
\cite{Nowak&Roncal}. In particular, from \cite[Theorem 2.4]{Nowak&Roncal} (see also comments in
\cite[Section 1]{Nowak&Roncal}) we get the following.
\begin{prop}  \label{prop:LpLq}
Let $\ab > -1$ and $\sigma > 0$. Assume that $p > p'(\ab)$ and $1 \le q < p(\ab)$. Then
$L_{\ab}^{-\sigma}$, $\alpha+\beta \neq -1$, and $(\id+L_{\ab})^{-\sigma}$ are bounded from
$L^p(0,\pi)$ to $L^q(0,\pi)$ if and only if
$$
\frac{1}q \ge \frac{1}{p} - 2\sigma.
$$
Moreover,
these operators are bounded from
$L^p(0,\pi)$ to $L^{\infty}(0,\pi)$ if and only if
$$
\ab\geq-1/2\quad\textrm{and}\quad \frac{1}{p}<2\sigma.
$$
\end{prop}

Following the classical picture, see e.g.\ \cite[Chapter V]{stein}, potential spaces in the Jacobi context
should be defined as the ranges of the Bessel type potentials acting on $L^p(0,\pi)$. However,
in our situation the spectrum of $L_{\ab}$ is discrete and separated from $0$ if $\alpha + \beta \neq -1$.
Therefore in case $\alpha+\beta \neq -1$ one can employ equivalently the Riesz type potentials, which
are simpler. Consequently, given $s>0$ and $p \in E(\ab)$ we set (see \cite{L2})
$$
\mathcal{L}_{\ab}^{p,s} := \begin{cases}
															L_{\ab}^{-s/2}\big( L^p(0,\pi)\big), & \alpha+\beta \neq -1,\\
															(\id+L_{\ab})^{-s/2}\big( L^p(0,\pi)\big), & \alpha+\beta = -1.
														\end{cases}
$$
Then the formula
$$
\|f\|_{\mathcal{L}_{\ab}^{p,s}} := \|g \|_{L^p(0,\pi)}, \qquad
																		\begin{cases} 
																			 f=L^{-s/2}_{\ab}g, \quad g \in L^p(0,\pi),
																			 		& \alpha+\beta \neq -1,\\														
																			 f = (\id+L_{\ab})^{-s/2}g, \quad g \in L^p(0,\pi),
																			 		& \alpha+\beta = -1,
																		\end{cases}
$$
defines a complete norm on $\mathcal{L}_{\ab}^{p,s}$. We call the resulting Banach spaces
$\mathcal{L}_{\ab}^{p,s}$ the \emph{Jacobi potential spaces}. 
Note that according to \cite[Corollary 2.6]{L2},
$S_{\ab}$ is a dense subspace of $\mathcal{L}_{\ab}^{p,s}$. 

In \cite{L2} the author introduced the \emph{Jacobi Sobolev spaces}
$$
W_{\ab}^{p,m} := \big\{ f \in L^p(0,\pi) : D^{(k)}f \in L^p(0,\pi), k=1,\ldots,m \big\},
$$
equipped with the norms
$$
\|f\|_{{W}_{\ab}^{p,m}} := \sum_{k=0}^{m} \|{D}^{(k)}f\|_{L^p(0,\pi)}.
$$
Here $m \ge 1$ is integer and the operators
$$
D^{(k)} := D_{\alpha+k-1,\beta+k-1} \circ \ldots \circ D_{\alpha+1,\beta+1} \circ D_{\ab}
$$
play the role of higher-order derivatives, with the differentiation understood in the weak sense.
The main result of \cite{L2} says that, for $\ab > -1$, $p \in E(\ab)$ and $m \ge 1$, we have
the coincidence $W_{\ab}^{p,m} = \mathcal{L}_{\ab}^{p,m}$ in the sense of isomorphism of Banach spaces.
A bit surprisingly, the isomorphism does not hold in general if $D^{(k)}$ is replaced by seemingly
more natural in this context $(D_{\ab})^k$.

We finish this preliminary section by invoking (see \cite[Section 2]{L2}) the following useful result, 
which is essentially a special case of the general multiplier-transplantation theorem due to Muckenhoupt
\cite[Theorem 1.14]{Muckenhoupt} (see \cite[Corollary 17.11]{Muckenhoupt} and also 
\cite[Theorem 2.5]{CNS} together with the related comments on pp{.}\,376--377 therein). 
Here and elsewhere we use the convention that $\phi_{n}^{\ab} \equiv 0$ if $n<0$.
\begin{lem}[Muckenhoupt]\label{mk}
Let $\alpha,\beta,\gamma,\delta>-1$ and let $d\in\Z$. 
Assume that $h(n)$ is a sequence satisfying for sufficiently large $n$ the smoothness condition
$$
h(n)=\sum_{j=0}^{J-1}c_{j}\,n^{-j}+\mathcal O(n^{-J}),
$$
where $J\ge \alpha+\beta+\gamma+\delta+6$ and $c_j$ are fixed constants.

Then for each $p$ satisfying $p'(\gamma,\delta) < p < p(\alpha,\beta)$ the operator 
$$
f\mapsto\sum_{n=0}^{\infty}h(n)\,a_n^{\ab}(f)\,\phi_{n+d}^{\gamma,\delta}(\theta),\qquad f\in S_{\ab},
$$ 
extends uniquely to a bounded operator on $L^p(0,\pi)$. 
\end{lem}

\section{Structural and embedding theorems} \label{sec:structural}

In this section we establish structural and embedding theorems for the Jacobi potential spaces. We begin
with recalling definitions of the variants of higher-order Riesz-Jacobi transforms considered in \cite{L2},
$$
R_{\ab}^{k}=\begin{cases} D^{(k)}L_{\ab}^{-k/2}, &\qquad \alpha+\beta\neq-1,\\
D^{(k)}(\id+L_{\ab})^{-k/2},&\qquad \alpha+\beta=-1.
\end{cases}
$$
Here $k \ge 0$ and $R_{\ab}^k$ are well defined at least on $S_{\ab}$. Using Lemma \ref{mk} it can be shown,
see \cite[Proposition 3.4]{L2}, that $R_{\ab}^k$ extend (uniquely) to bounded operators on $L^p(0,\pi)$,
$p \in E(\ab)$, $\ab > -1$.

The following result reveals mutual relations between Jacobi potential spaces with different parameters.
It also describes mapping properties of the Riesz-Jacobi transforms acting on the potential spaces.
\begin{thm} \label{thm:struc}
Let $\ab > -1$ and $p \in E(\ab)$. Assume that $r,s > 0$ and $k \ge 1$.
\begin{itemize}
\item[(i)] If $r < s$, then $\mathcal{L}_{\ab}^{p,s} \subset \mathcal{L}_{\ab}^{p,r} \subset L^p(0,\pi)$
	and the inclusions are proper and continuous.
\item[(ii)] The spaces $\mathcal{L}_{\ab}^{p,r}$ and $\mathcal{L}_{\ab}^{p,s}$ are isometrically isomorphic.
\item[(iii)] If $k<s$, then $D^{(k)}$ is bounded from $\mathcal{L}_{\ab}^{p,s}$ to
	$\mathcal{L}_{\alpha+k,\beta+k}^{p,s-k}$. Moreover, $D^{(k)}$ is bounded from $\mathcal{L}_{\ab}^{p,k}$
	to $L^p(0,\pi)$.
\item[(iv)] The Riesz operator $R_{\ab}^k$ is bounded from $\mathcal{L}_{\ab}^{p,s}$ to
	$\mathcal{L}_{\alpha+k,\beta+k}^{p,s}$. 
\end{itemize}
\end{thm}

\begin{proof}
Throughout the proof we assume that $\alpha+\beta \neq -1$. The opposite case is essentially parallel
(with (iii) and (iv) requiring a little bit more attention) and thus is left to the reader.

We first prove (i). Take $f \in \mathcal{L}_{\ab}^{p,s}$. Then, 
by the definition of $\mathcal{L}_{\ab}^{p,s}$, there exists $g \in L^p(0,\pi)$ such that
$f = L_{\ab}^{-s/2}g$. But this identity can be written as
$$
f = L_{\ab}^{-r/2} \big(L_{\ab}^{-(s-r)/2}g\big).
$$
Indeed, the equality 
\begin{equation} \label{comp}
L_{\ab}^{-s/2}g = L_{\ab}^{-r/2}(L_{\ab}^{-(s-r)/2}g), \qquad 0< r < s,
\end{equation}
is clear when $g \in S_{\ab}$,
and then for $g \in L^p(0,\pi)$ it follows by an approximation argument and 
$L^p$-boundedness of the potential operators.
Now, since Proposition \ref{prop:LpLq} implies $L_{\ab}^{-(s-r)/2}g \in L^p(0,\pi)$, we conclude that
$f \in \mathcal{L}_{\ab}^{p,r}$. Moreover, the inclusion just proved is continuous because
$L_{\ab}^{-(s-r)/2}$ is bounded on $L^p(0,\pi)$. The remaining inclusion is even more straightforward,
in view of the $L^p$-boundedness of $L_{\ab}^{-r/2}$. 
The fact that the reverse inclusions do not hold is verified as follows. 

Observe that, in view of the inclusions already proved, it suffices to show that
$\mathcal{L}_{\ab}^{p,r} \neq \mathcal{L}_{\ab}^{p,s}$ when $0 < r < s$ are rational numbers.
This task further reduces to proving that
\begin{equation} \label{red}
\mathcal{L}_{\ab}^{p,r} \neq L^p(0,\pi), \qquad 0 < r \in \mathbb{Q}.
\end{equation}
Indeed, suppose on the contrary that $\mathcal{L}_{\ab}^{p,r} = \mathcal{L}_{\ab}^{p,s}$. Then, for
any $f \in L^p(0,\pi)$ we have $L_{\ab}^{-r/2}f \in \mathcal{L}_{\ab}^{p,r}$ and so there is
$g \in L^p(0,\pi)$ such that $L_{\ab}^{-r/2}f = L_{\ab}^{-s/2}g = L_{\ab}^{-r/2} L_{\ab}^{-(s-r)/2}g$,
see \eqref{comp}. Since the Riesz potentials are injective (see \cite[Proposition 2.4]{L2}), it follows
that $f = L_{\ab}^{-(s-r)/2}g$. This implies $f \in \mathcal{L}_{\ab}^{p,s-r}$ and, consequently,
$\mathcal{L}_{\ab}^{p,s-r} = L^p(0,\pi)$. A contradiction with \eqref{red}.

It remains to justify \eqref{red}. Suppose that $\mathcal{L}_{\ab}^{p,r} = L^p(0,\pi)$ for some
rational $r > 0$. We will derive a contradiction. Take $1 \le m \in \mathbb{N}$ such that $mr$ is
integer and pick an arbitrary $f \in L^p(0,\pi)$. Then, taking into account what we have assumed, 
$f \in \mathcal{L}_{\ab}^{p,r}$ and so there is $g_1 \in L^p(0,\pi)$ such that $f = L_{\ab}^{-r/2}g_1$.
Similarly, we can find $g_2 \in L^p(0,\pi)$ such that $g_1 = L_{\ab}^{-r/2}g_2$. Iterating this procedure
we get, see \eqref{comp}, $f = (L_{\ab}^{-r/2})^m g_m = L_{\ab}^{-mr/2}g_m$ for some $g_m \in L^p(0,\pi)$.
Consequently, $f \in \mathcal{L}_{\ab}^{p,mr}$. According to \cite[Theorem A]{L2}, 
$\mathcal{L}_{\ab}^{p,mr} = W_{\ab}^{p,mr}$, the Jacobi Sobolev space. We conclude that
$L^p(0,\pi) = W_{\ab}^{p,mr}$. This means, in particular, that $D_{\ab}f \in L^p(0,\pi)$ for each
$f \in L^p(0,\pi)$. But the latter is false, as can be easily seen by taking either $f \equiv 1$
in case $(\ab) \neq (-1/2,-1/2)$ or $f(\theta) = \log \theta$ otherwise. The desired contradiction follows. 

To show (ii) we may assume, for symmetry reasons, that $r < s$. Then it is straightforward to see that
the operator
$$
L_{\ab}^{-(s-r)/2} \colon \mathcal{L}_{\ab}^{p,r} \longrightarrow \mathcal{L}_{\ab}^{p,s}
$$
is an isometric isomorphism, see \eqref{comp}.

We pass to showing (iii). Observe that it is enough to treat the case $k=1$, since then the general case
is obtained by simple iterations. To see that $D_{\ab}$ is bounded from $\mathcal{L}_{\ab}^{p,s}$ to
$\mathcal{L}_{\alpha+1,\beta+1}^{p,s-1}$ for $s>1$, it suffices to prove that
$$
\big\| L_{\alpha+1,\beta+1}^{(s-1)/2} D_{\ab} L_{\ab}^{-s/2} g \big\|_p \lesssim \|g\|_p, 
	\qquad g \in S_{\ab}.
$$
Taking into account the identities
$$
D_{\ab} \phi_{n}^{\ab} = - \sqrt{\lambda_n^{\ab}-\lambda_0^{\ab}} \, \phi_{n-1}^{\alpha+1,\beta+1},
$$
see \cite[(5)]{L2}, and $\lambda_n^{\ab} = \lambda_{n-1}^{\alpha+1,\beta+1}$, $n \ge 1$, we write
$$
L_{\alpha+1,\beta+1}^{(s-1)/2} D_{\ab} L_{\ab}^{-s/2} g = - \sum_{n=1}^{\infty} \bigg( \frac{\lambda_n^{\ab}-\lambda_0^{\ab}}{\lambda_n^{\ab}}\bigg)^{1/2}
 a_n^{\ab}(g)
	 \phi_{n-1}^{\alpha+1,\beta+1},
		\qquad g \in S_{\ab}.
$$
Now an application of Lemma \ref{mk} leads directly to the desired conclusion. The fact that the function
$h(n) = (1-\lambda_0^{\ab}/\lambda_{n}^{\ab})^{1/2}$ indeed satisfies the assumptions of Lemma \ref{mk}
is verified by arguments analogous to those in the proof of \cite[Proposition 3.4]{L2}.

Finally, (iv) is a consequence of (iii) and the fact that $L_{\ab}^{-k/2}$ is bounded from
$\mathcal{L}_{\ab}^{p,s}$ to $\mathcal{L}_{\ab}^{p,s+k}$.
\end{proof}

Our next result corresponds to the classical embedding theorem due to Sobolev (the latter can be found,
for instance, in \cite[Chapter V]{stein}). 
Recall that for integer values of $s$, say $s=m$, the potential spaces
$\mathcal{L}_{\ab}^{p,m}$ coincide with the Jacobi Sobolev spaces $W_{\ab}^{p,m}$ investigated in \cite{L2}.
\begin{thm} \label{thm:embed}
Let $\ab > -1$, $p \in E(\ab)$ and $1\le q < p(\ab)$.
\begin{itemize}
\item[(i)]
If $s > 0$ is such that $1/q > 1/p -s$, then $\mathcal{L}_{\ab}^{p,s} \subset L^q(0,\pi)$ and
\begin{equation} \label{embest}
\|f\|_{q} \lesssim \|f\|_{\mathcal{L}_{\ab}^{p,s}}, \qquad f \in \mathcal{L}_{\ab}^{p,s}.
\end{equation}
\item[(ii)]
If $\ab \ge -1/2$ and $s > 1/p$, then $\mathcal{L}_{\ab}^{p,s} \subset C(0,\pi)$ and \eqref{embest}
holds with $q=\infty$.
\end{itemize}
\end{thm}

\begin{proof}
We assume that $\alpha+\beta \neq -1$, the opposite case is analogous.
Let $f \in \mathcal{L}_{\ab}^{p,s}$. Then there exists $g \in L^p(0,\pi)$ such that $f = L_{\ab}^{-s/2}g$.
According to Proposition \ref{prop:LpLq}, the potential operator $L_{\ab}^{-s/2}$ is of strong type
$(p,q)$ for $p$ and $q$ admitted in (i) and (ii) (to be precise, in (ii) $q=\infty$).
Thus $f \in L^q(0,\pi)$ and \eqref{embest} holds. 

It remains to show that, under the assumptions of (ii), $f$ is continuous.
Since $S_{\ab}$ is a dense subspace of $\mathcal{L}_{\ab}^{p,s}$, there
exists a sequence $\{f_n\}\subset S_{\ab}$ such that $f_n \to f$ in $\mathcal{L}_{\ab}^{p,s}$. Then
$$
\|f-f_n\|_{\infty} \lesssim \|f-f_n\|_{\mathcal{L}_{\ab}^{p,s}} \to 0, \qquad n \to \infty,
$$
and we see that $f$ is a uniform limit of continuous functions.
\end{proof}

\section{Characterization by fractional square functions} \label{sec:char}

Let $\ab > -1$. Following Betancor et.\ al.\ \cite{betsq}, we consider a pair of fractional square
functions
\begin{align*}
\mathfrak g_{\ab}^{\gamma}(f)(\theta) & = \bigg(\int_0^{\infty}\big|t^{\gamma}\partial_t^{\gamma}
	 H_t^{\ab}f(\theta)\big|^2\frac{dt}{t}\bigg)^{1/2}, \qquad \gamma > 0, \\
\mathfrak g_{\ab}^{\gamma,k}(f)(\theta) & =\bigg(\int_0^{\infty}\Big|t^{k-\gamma}
	\frac{\partial^k}{\partial t^k} H_t^{\ab}f(\theta)\Big|^2\frac{dt}{t}\bigg)^{1/2},
		\qquad 0<\gamma<k, \quad k \in \mathbb{N}.
\end{align*}
Here $\partial_t^{\gamma}$ denotes a Caputo type fractional derivative given, for suitable $F$, by
\begin{equation}\label{fd}
\partial_t^{\gamma}F(t)=\frac{1}{\Gamma(m-\gamma)}\int_0^{\infty}
	\frac{\partial^m}{\partial t^m}F(t+s)\,s^{m-\gamma-1}\,ds,\qquad t> 0,
\end{equation}
where $m=\lfloor\gamma\rfloor+1$, $\lfloor\cdot\rfloor$ being the floor function.
The study of square functions involving $\partial_t^{\gamma}$ goes back to Segovia and Wheeden
\cite{Seg&Whe}, where the classical setting was considered.

Note that $\mathfrak{g}_{\ab}^{\gamma}(f)$ and $\mathfrak{g}_{\ab}^{\gamma,k}(f)$ are well defined pointwise
for $f \in L^p(0,\pi)$, $p > p'(\ab)$. This is clear in case of $\mathfrak{g}_{\ab}^{\gamma,k}$, since
$H_t^{\ab}$ is smooth in $t>0$. To see this property for $\mathfrak{g}_{\ab}^{\gamma}$, we observe that
$\partial_t^{\gamma}H_t^{\ab}f$ is well defined pointwise if $f$ is as above. In fact
\begin{equation} \label{dgHt}
\partial_t^\gamma H_t^{\ab}f(\theta)= (-1)^m\sum_{n=0}^{\infty}
 \big(\lambda_n^{\ab}\big)^{\gamma/2}\exp\Big(-t\sqrt{\lambda_{n}^{\ab}}\Big)\,a_n^{\ab}(f)\,
  \phi_n^{\ab}(\theta),
\end{equation}
and the series converges for each $t > 0$ and $\theta \in (0,\pi)$. This follows by term-by-term
differentiation and integration of the series defining $H_t^{\ab}f$. 
Such manipulations are indeed legitimate, as can be easily checked with the aid of \eqref{bphi} and
the resulting polynomial growth in $n$ of $a_n^{\ab}(f)$.

The first main result of this section is the following characterization of the Jacobi potential 
spaces in terms of $\mathfrak{g}_{\ab}^{\gamma,k}$.
\begin{thm}\label{equivfun}
Let $\ab > -1$, $p\in E(\ab)$ and assume that $\alpha+\beta\neq-1$. 
Fix $0 < \gamma < k$ with $k\in \mathbb{N}$.
Then $f\in\mathcal{L}_{\ab}^{p,\gamma}$ if and only if $f\in L^p(0,\pi)$ and
$\mathfrak{g}_{\ab}^{\gamma,k}(f)\in L^p(0,\pi)$. Moreover,
$$
\|f\|_{\mathcal{L}_{\ab}^{p,\gamma}}  \simeq
\big\|\mathfrak{g}_{\ab}^{\gamma,k}(f)\big\|_{p}, \qquad f\in \mathcal{L}_{\ab}^{p,\gamma}.
$$
\end{thm}

\begin{remark}\label{uwaga}
To get a similar characterization in the singular case $\alpha+\beta=-1$ one has to modify suitably
the square function $\mathfrak{g}_{\ab}^{\gamma,k}$. The corresponding statement can be found at the end
of this section, see Theorem \ref{equivfun'}.
\end{remark}

To prove Theorem \ref{equivfun} we follow a general strategy presented in \cite{betsq}.
The main difficulty in this approach is showing that the fractional square function 
$\mathfrak{g}_{\ab}^{\gamma}$ preserves $L^p$ norms, as stated below.

\begin{thm}\label{lpequivf}
Let $\ab>-1$, $p\in E(\ab)$ and $\gamma>0$. Then
\begin{equation*}
\|f\|_p\simeq\big\|\mathfrak{g}_{\ab}^{\gamma}(f)\big\|_p+\chi_{\{\alpha+\beta=-1\}}\,\big|a_0^{\ab}(f)\big|,
	\qquad f\in L^p(0,\pi).
\end{equation*}
\end{thm}
For the time being, in this section we assume that Theorem \ref{lpequivf} holds and postpone its proof
until Section \ref{sec:g-functions}. Then to show Theorem \ref{equivfun} it suffices to 
ensure that the general arguments in \cite{betsq} work when specified to the Jacobi framework.
We begin with two auxiliary results which appear almost explicitly in \cite{betsq}. 

\begin{lem}\label{2.6}
Let $\ab>-1$, $p\in E(\ab)$ and assume that $0<\gamma<k\leq l$ with $k, l\in\mathbb{N}$. 
Then
$$
\big\|\mathfrak g_{\ab}^{\gamma, l}(f)\big\|_{p}\lesssim\big\|\mathfrak g_{\ab}^{\gamma, k}(f)\big\|_{p},
 \qquad f\in L^p(0,\pi).
$$
\end{lem}

\begin{proof}
We use the $L^p$-boundedness of $\mathfrak{g}_{\ab}^1$ (see Theorem \ref{lpequivf}) and repeat
the arguments from the proof of \cite[Proposition 2.6]{betsq}. Everything indeed works for general
$f \in L^p(0,\pi)$ thanks to the smoothness of $H_t^{\ab}f$ in $t >0$.
\end{proof}

\begin{lem}\label{lemma2.2}
Let $\ab>-1$, $\alpha+\beta\neq-1$, $p\in E(\ab)$ and $0<\gamma<k$ with $k \in \mathbb{N}$.
Then $\mathfrak g_{\ab}^{\gamma, k}$ is bounded on $\mathcal{L}_{\ab}^{p,\gamma}$. 
Furthermore,
$$
\mathfrak g_{\ab}^{\gamma, k}(f)=\mathfrak g_{\ab}^{k-\gamma}\big(L_{\ab}^{\,\gamma/2}f\big), 
	\qquad f\in \mathcal{L}_{\ab}^{p,\gamma},
$$
with $L_{\ab}^{\gamma/2}$ understood as the inverse of the potential operator $L_{\ab}^{-\gamma/2}$.
\end{lem}

\begin{proof}
In view of \cite[Lemma 2.2 (ii)]{betsq}, the identity 
$\mathfrak g_{\ab}^{\gamma, k}(f)=\mathfrak g_{\ab}^{k-\gamma}(L_{\ab}^{\,\gamma/2}f)$ holds for
$f\in S_{\ab}$. Taking into account that $\mathcal{L}_{\ab}^{p,\gamma} = L_{\ab}^{-\gamma/2}(L^p(0,\pi))$
and $L_{\ab}^{-\gamma/2}$ is one-to-one, $S_{\ab}$ is dense in $\mathcal{L}_{\ab}^{p,\gamma}$ 
and $\mathfrak{g}_{\ab}^{k-\gamma}$ is bounded on $L^p(0,\pi)$
(see Theorem \ref{lpequivf}), we arrive at the desired conclusion.
\end{proof}

Lemma \ref{lemma2.2} together with Theorem \ref{lpequivf} implies the equivalence of norms asserted
in Theorem~\ref{equivfun}, which we state as the following.
\begin{prop}\label{eq}
Let $\ab>-1$, $\alpha+\beta\neq-1$, $p\in E(\ab)$ and $0<\gamma<k$ with $k\in\mathbb{N}$. 
Then 
\begin{equation*}
\|f\|_{\mathcal{L}_{\ab}^{p,\gamma}}\simeq\big\|\mathfrak g_{\ab}^{\gamma,k}(f)\big\|_p,
	\qquad f\in \mathcal{L}_{\ab}^{p,\gamma}.
\end{equation*}
\end{prop}

We are now in a position to prove Theorem \ref{equivfun}. We follow the line of reasoning 
from the proof of \cite[Proposition 4.1]{betsq}.
\begin{proof}[Proof of Theorem \ref{equivfun}]
In view of Proposition \ref{eq}, what we need to prove is that $f \in \mathcal{L}_{\ab}^{p,\gamma}$ if
$f \in L^p(0,\pi)$ and $\mathfrak{g}_{\ab}^{\gamma,k}(f) \in L^p(0,\pi)$. Thus we assume that
$f,\mathfrak{g}_{\ab}^{\gamma,k}(f) \in L^p(0,\pi)$. 

Let 
$$
F_t= \sum_{n=0}^{\infty} \big(\lambda_n^{\ab}\big)^{\gamma/2} \exp\Big(-t\sqrt{\lambda_{n}^{\ab}}\Big)
	a_n^{\ab}(f) \phi_n^{\ab}, \qquad t > 0.
$$
Notice that $F_t = (-1)^m \partial_t^{\gamma}H_t^{\ab}f$, see \eqref{dgHt}.
The series defining $F_t$ converges in $L^p(0,\pi)$, as can be easily verified by means of \eqref{bphi}.
Since the potential operator $L_{\ab}^{-\gamma/2}$ is $L^p$-bounded,
we have $L_{\ab}^{-\gamma/2}F_t = H_t^{\ab}f$ and, consequently, 
$H_t^{\ab}f \in \mathcal{L}_{\ab}^{p,\gamma}$ for $t >0$.

Next, let $l\in\mathbb{N}$ be such that $l>k$ and $l > \gamma + 1/2$. 
By Proposition \ref{eq} one has
$$
\|F_t\|_p=\big\|H_{t}^{\ab}f\big\|_{\mathcal{L}_{\ab}^{p,\gamma}} \simeq 
	\big\|\mathfrak {g}_{\ab}^{\gamma,l}(H_{t}^{\ab}f)\big\|_p, \qquad f \in L^p(0,\pi), \quad t > 0.
$$
Further, exploiting the semigroup property of $\{H_{t}^{\ab}\}$ we get, for $\theta \in (0,\pi)$,
\begin{align*}
\big|\mathfrak {g}_{\ab}^{\gamma,l}(H_{t}^{\ab}f)(\theta)\big|^2 
&=\int_0^{\infty}\Big|t^{l-\gamma}\frac{\partial^l}{\partial s^l}H_{t+s}^{\ab}f(\theta)\Big|^2\,\frac{ds}{s}
\\
& \le \int_0^{\infty}\Big|(t+s)^{l-\gamma}
	\frac{\partial^l}{\partial s^l}H_{t+s}^{\ab}f(\theta)\Big|^2\,\frac{ds}{t+s}\\
&\le \int_0^{\infty}\Big|s^{l-\gamma}\frac{\partial^l}{\partial s^l}H_{s}^{\ab}f(\theta)\Big|^2\,\frac{ds}{s}
\\
& =\big|\mathfrak g_{\ab}^{\gamma,l}(f)(\theta)\big|^2.
\end{align*}
Combining the above with Lemma \ref{2.6} we obtain
$$
\|F_t\|_p\lesssim \big\|\mathfrak g_{\ab}^{\gamma,k}(f)\big\|_p, \qquad f \in L^p(0,\pi), \quad t>0.
$$

Now, by the Banach-Alaoglu theorem there exists a decreasing positive sequence $t_n \to 0$ and a function
$F\in L^p(0,\pi)$ such that $F_{t_n}\to F$ in the weak* topology of $L^p(0,\pi)$. 
Then, since $L_{\ab}^{-\gamma/2}$ is $L^p$-bounded, we also have
$$
H_{t_n}^{\ab}f=L_{\ab}^{-\gamma/2}F_{t_n}\to L_{\ab}^{-\gamma/2}F
$$
in the weak* topology of $L^p(0,\pi)$. 
On the other hand, $H_{t_n}^{\ab}f\to f$ in $L^p(0,\pi)$, which follows by the $L^p$-boundedness of the
maximal operator $f \mapsto \sup_{t>0} |H_t^{\ab}f|$ (see \cite[Proposition 2.2]{L2}) and the density of
$S_{\ab}$ in $L^p(0,\pi)$. We conclude that $f=L_{\ab}^{-\gamma/2}F$, which means that
$f\in\mathcal{L}_{\ab}^{p,\gamma}$.
\end{proof} 

We now come back to the issue of characterizing $\mathcal{L}_{\ab}^{p,\gamma}$ when $\alpha+\beta=-1$.
Actually, by means of a variant of $\mathfrak{g}_{\ab}^{\gamma,k}$, we will characterize the Jacobi potential
spaces for any $\ab > -1$, see Theorem \ref{equivfun'} below. This is the second main result of this section.

Let $\ab > -1$ and $\gamma > 0$. Consider the modified Jacobi Laplacian
$$
\widetilde{L}_{\ab}:=\big(\id + \sqrt{L_{\ab}}\big)^2
$$ 
and the related modified Bessel type potentials $\widetilde{L}_{\ab}^{-\gamma/2}$. 
Clearly, the latter operators are well defined spectrally and bounded
on $L^2(0,\pi)$. Moreover, Lemma \ref{mk} shows that they extend uniquely to bounded operators
on $L^p(0,\pi)$, $p \in E(\ab)$ (we keep the same notation for these extensions). Furthermore,
similarly as in the proof of \cite[Proposition 2.4]{L2}, one can verify that 
$\widetilde{L}_{\ab}^{-\gamma/2}$ is one-to-one on $L^p(0,\pi)$, $p \in E(\ab)$.
Thus we can define alternative potential spaces via the modified Bessel type potentials,
$$
\widetilde{\mathcal{L}}_{\ab}^{p,\gamma} := \widetilde{L}_{\ab}^{-\gamma/2}\big(L^p(0,\pi)\big),
	\qquad p \in E(\ab),
$$
normed by $\|f\|_{\widetilde{\mathcal{L}}_{\ab}^{p,\gamma}} := \|g\|_{p}$, where
$f = \widetilde{L}_{\ab}^{-\gamma/2}g$. These are Banach spaces, and the crucial fact is that they are
isomorphic to $\mathcal{L}_{\ab}^{p,\gamma}$. More precisely, $\mathcal{L}_{\ab}^{p,\gamma}$ and
$\widetilde{\mathcal{L}}_{\ab}^{p,\gamma}$ coincide as sets of functions and the two norms are equivalent.
To see this, it is enough to observe that the multiplier operators
$$
\frac{(\id + L_{\ab})^{\gamma/2}}{(\id + \sqrt{L_{\ab}})^{\gamma}}, \qquad
\frac{(\id + \sqrt{L_{\ab}})^{\gamma}}{(\id + L_{\ab})^{\gamma/2}},
$$
being mutual inverses defined initially on $L^2(0,\pi)$, both extend to bounded operators on
$L^p(0,\pi)$, $p \in E(\ab)$. The latter follows readily by means of Lemma \ref{mk}.

The Poisson semigroup corresponding to $\widetilde{L}_{\ab}$ is generated by $-\id-\sqrt{L_{\ab}}$,
hence it has the form $\{e^{-t}H_t^{\ab}\}$. Consequently, the relevant fractional square functions
are given by
\begin{align*}
\widetilde{\mathfrak g}_{\ab}^{\,\gamma}(f)(\theta) & 
= \bigg(\int_0^{\infty}\big|t^{\gamma}\partial_t^{\gamma}
	 \big[e^{-t}H_t^{\ab}f(\theta)\big]\big|^2\frac{dt}{t}\bigg)^{1/2}, \qquad \gamma > 0, \\
\widetilde{\mathfrak g}_{\ab}^{\,\gamma,k}(f)(\theta) & =\bigg(\int_0^{\infty}\Big|t^{k-\gamma}
	\frac{\partial^k}{\partial t^k} \big[e^{-t}H_t^{\ab}f(\theta)\big]\Big|^2\frac{dt}{t}\bigg)^{1/2},
		\qquad 0<\gamma<k, \quad k \in \mathbb{N}.
\end{align*}
A reasoning parallel to that in Section \ref{sec:g-functions} shows that, given $p \in E(\ab)$,
$$
\big\|\widetilde{\mathfrak g}_{\ab}^{\,\gamma}(f)\big\|_p \simeq \|f\|_{p}, \qquad f \in L^p(0,\pi).
$$

All the above facts and a direct adaptation of the ingredients and arguments proving Theorem~\ref{equivfun}
lead to the following alternative characterization of $\mathcal{L}_{\ab}^{p,\gamma}$, 
valid for all $\ab > -1$.

\begin{thm}\label{equivfun'}
Let $\ab > -1$ and $p\in E(\ab)$. Fix $0 < \gamma < k$ with $k \in \mathbb{N}$.
Then $f\in\mathcal{L}_{\ab}^{p,\gamma}$ if and only if $f\in L^p(0,\pi)$ and
$\widetilde{\mathfrak{g}}_{\ab}^{\,\gamma,k}(f)\in L^p(0,\pi)$. Moreover,
$$
\|f\|_{\mathcal{L}_{\ab}^{p,\gamma}} \simeq
	\big\|\widetilde{\mathfrak{g}}_{\ab}^{\,\gamma,k}(f)\big\|_{p},
		\qquad f\in \mathcal{L}_{\ab}^{p,\gamma}.
$$
\end{thm}

\begin{proof}
This is a repetition of the arguments already presented.
We leave details to interested readers.
\end{proof}

\section{Sample applications of the potential spaces} \label{sec:convergence}

The first application we present is motivated by the results in \cite[Section 7]{BT1} and
\cite[Section 6]{BT2}, see also references therein.
Given some initial data $f \in L^2(0,\pi)$, consider the following Cauchy problem based on the Jacobi
Laplacian:
$$
\begin{cases}
\big(i\partial_t + L_{\ab}\big) u(\theta,t)=0 \\
u(\theta,0)=f(\theta)
\end{cases},
\qquad \theta\in(0,\pi),\quad t\in\mathbb{R}.
$$
It is straightforward to check that $\exp(itL_{\ab})f$ is a solution to this problem
(here $\exp(itL_{\ab})$ is understood spectrally). Then a natural and important question is the following:
what regularity conditions should be imposed on $f$ to guarantee pointwise almost everywhere convergence
of the solution to the initial condition? It turns out that a sufficient condition for this convergence
can be stated in terms of the Jacobi potential spaces.
 
\begin{prop}
Let $\ab > -1$ and $s > 1/2$. Then for each $f\in \mathcal{L}_{\ab}^{2,s}$
$$
\lim_{t\rightarrow0}\,\exp(itL_{\ab})f(\theta)=f(\theta) \qquad \textrm{a.a.}\; \theta \in (0,\pi).
$$
\end{prop}

\begin{proof}
In the proof we assume that $\alpha+\beta \neq -1$; the opposite case requires obvious modifications,
which are left to the reader. Let $f \in \mathcal{L}_{\ab}^{2,s}\subset L^2(0,\pi)$ and observe that
$\exp(itL_{\ab})f$ is well defined in the $L^2$ sense. It is straightforward to check that
$$
\lim_{t \to 0} \exp(itL_{\ab})f(\theta) = f(\theta), \qquad \theta \in (0,\pi), \quad f \in S_{\ab}.
$$
Recall that $S_{\ab}$ is a dense subspace of $\mathcal{L}_{\ab}^{2,s}$.

We will show that the set
$$
A=\big\{\theta\in(0,\pi):\limsup_{t\to 0}|\exp(itL_{\ab})f(\theta)-f(\theta)|>0\big\}
$$
has Lebesgue measure zero. Denote
$I_N=[\frac{1}{N},\pi-\frac{1}{N}]$ and 
$$
A_{N,k}=\Big\{\theta\in I_N:\limsup_{t\to 0}|\exp(itL_{\ab})f(\theta)-f(\theta)|>\frac{1}{k}\Big\}.
$$
Since the sum of $A_{N,k}$ over $N,k \ge 1$ gives $A$, it is enough to prove that $|A_{N,k}|=0$
for each $N$ and $k$ fixed.

To proceed, we consider the maximal operator
$$
T_*f(\theta)=\sup_{t\in\mathbb{R}}|\exp(itL_{\ab})f(\theta)|.
$$ 
We have
$$
\int_{I_N} T_{*}f(\theta)\, d\theta \le \sum_{n=0}^{\infty} \big| a_n^{\ab}(f)\big| \int_{I_N}
	|\phi_n^{\ab}(\theta)|\, d\theta.
$$
The integrals here can be bounded by means of the estimate, see \cite[Theorem 8.21.8]{Sz},
$$
|\phi_n^{\ab}(\theta)| \le C_N, \qquad \theta \in I_N, \quad n \ge 0
$$
(the constant $C_N$ depends on $N$ and possibly also on $\alpha$ and $\beta$). 
Then, using Schwarz' inequality,
we get
\begin{equation} \label{max}
\int_{I_N} T_{*}f(\theta)\, d\theta \le C_N \bigg( \sum_{n=0}^{\infty} \big| \big(\lambda_n^{\ab}\big)^{s/2}
	a_n^{\ab}(f)\big|^2 \bigg)^{1/2} \bigg( \sum_{n=0}^{\infty} \big(\lambda_n^{\ab}\big)^{-s}\bigg)^{1/2}
		= C'_N \|f\|_{\mathcal{L}_{\ab}^{2,s}},
\end{equation}
where $C'_N$ depends also on $s$.

Now we are ready to show that $|A_{N,k}|=0$. Take $0<\varepsilon<1$ and choose $f_0\in S_{\ab}$
such that $\|f-f_0\|_{\mathcal{L}_{\ab}^{2,s}}< \varepsilon$. We have
$A_{N,k} \subset A_{N,k}^1 \cup A_{N,k}^2 \cup A_{N,k}^3$, where
\begin{align*}
A_{N,k}^1 & = \Big\{ \theta \in I_N : |f(\theta)-f_0(\theta)| > \frac{1}{3k} \Big\}, \\
A_{N,k}^2 & = \Big\{ \theta \in I_N : \limsup_{t \to 0} \big|\exp(itL_{\ab})f_0(\theta)-f_0(\theta)\big| > 
	\frac{1}{3k} \Big\},\\
A_{N,k}^3 & = \Big\{ \theta \in I_N : \limsup_{t \to 0} \big|\exp(itL_{\ab})f(\theta)
	-\exp(itL_{\ab})f_0(\theta)\big| > \frac{1}{3k} \Big\}.
\end{align*}
Notice that $A_{N,k}^2 = \emptyset$. For $A_{N,k}^1$ we write
\begin{align*}
|A_{N,k}^1| & \le (3k)^2 \int_{I_N} |f(\theta)-f_0(\theta)|^2\, d\theta
	\le (3k)^2 \|f-f_0\|_2^2 \le (3k)^2 \big|\lambda_0^{\ab}\big|^{-s} \|f-f_0\|^2_{\mathcal{L}_{\ab}^{2,s}} \\
&		< (3k)^2 \big|\lambda_0^{\ab}\big|^{-s}\, \varepsilon,
\end{align*}
where we used the equality $\big\|L_{\ab}^{-s/2}\big\|_{L^2\to L^2} = \big|\lambda_0^{\ab}\big|^{-s/2}$.
Finally, to deal with $A_{N,k}^3$ we use \eqref{max} and obtain
$$
|A_{N,k}^3| \le 3k \int_{I_N} T_{*}(f-f_0)(\theta)\, d\theta \le
	3k C'_N \|f-f_0\|_{\mathcal{L}_{\ab}^{2,s}} < 3kC'_N\, \varepsilon.
$$
Since we can choose $\varepsilon$ arbitrarily small, it follows that $|A_{N,k}|=0$
\end{proof}

Another result 
involving the Jacobi potential spaces is the following mixed norm smoothing estimate
motivated by the results of \cite[Section 3]{BR}.
\begin{prop}\label{strichartz}
Let $\ab > -1$ and $p \in E(\ab)$. Assume that $s>0$ is such that $s \ge 1/2 + \max\{\alpha,\beta,-1/2\}$
and $\alpha+\beta$ is integer. Then
\begin{equation*}
\big\|\exp(itL_{\ab})f\big\|_{L_{\theta}^p((0,\pi),\,L_t^2(0,2\pi))}\lesssim \|f\|_{\mathcal L_{\ab}^{2,s}},
\qquad f \in \mathcal{L}_{\ab}^{2,s}.
\end{equation*}
\end{prop}

\begin{proof}
Throughout the proof we assume that $\alpha+\beta \neq -1$, since the opposite case requires only
minor modifications.
By a density argument it suffices to prove the asserted bound for $f\in S_{\ab}$. 
For such $f$ we have
\begin{align*}
\big\|\exp(itL_{\ab})f\big\|_{L_t^2(0,2\pi)}^2&=\int_0^{2\pi}\Big(\sum\limits_{n=0}^{\infty}
e^{it\lambda_n^{\ab}}\,a_n^{\ab}(f)\,\ph\Big)
\Big(\sum\limits_{n=0}^{\infty}e^{-it\lambda_n^{\ab}}\,\overline{a_n^{\ab}(f)}\,\ph\Big)\,dt\\
&=2\pi\,\sum\limits_{n=0}^{\infty}\big|a_n^{\ab}(f)\big|^2\,\big(\ph\big)^2,
\end{align*}
since $\lambda_n^{\ab}-\lambda_m^{\ab} = (n-m)(n+m+\alpha+\beta+1)$ is integer. Then
applying Minkowski's inequality we get
\begin{equation*}
\|\exp(itL_{\ab})f\|_{L_{\theta}^p((0,\pi),\,L_t^2(0,2\pi))}\le
 \sqrt{2\pi} \bigg(\sum\limits_{n=0}^{\infty}\big|a_n^{\ab}(f)\big|^2\,\big\|\ph\big\|_{p}^2\bigg)^{1/2}.
\end{equation*}
By means of \eqref{bphi} we can estimate the $L^p$ norms here,
$$
\big\|\ph\big\|_p \lesssim \big\|\Psi^{\ab}\big\|_p \, 
	(n+1)^{1/2+\max\{\alpha,\beta,-1/2\}} \lesssim (n+1)^s,\qquad n \ge 0.
$$
Applying now Parseval's identity we arrive at
\begin{align*}
\|\exp(itL_{\ab})f\|_{L_{\theta}^p((0,\pi),\,L_t^2(0,2\pi))}& \lesssim \bigg( \sum_{n=0}^{\infty} (n+1)^{2s}
	\big|a_n^{\ab}(f)\big|^2 \bigg)^{1/2} \\
& \lesssim \bigg( \sum_{n=0}^{\infty} \big(\lambda_n^{\ab}\big)^{s} \big|a_n^{\ab}(f)\big|^2 \bigg)^{1/2}\\
& = \bigg\| \sum_{n=0}^{\infty} \big(\lambda_n^{\ab}\big)^{s/2} a_n^{\ab}(f)\, \phi_n^{\ab} \bigg\|_2 
 = \|f\|_{\mathcal{L}_{\ab}^{2,s}}.
\end{align*}
This finishes the proof.
\end{proof}

Finally, we give an extension of Proposition \ref{strichartz}.
\begin{prop} \label{prop:ext}
Let $\ab, p$ and $s$ be as in Proposition \ref{strichartz} and assume that $q>2$. 
Then
$$
\big\|\exp(itL_{\ab})f\big\|_{L_{\theta}^p((0,\pi),\,L_t^q(0,2\pi))} \lesssim
	\|f\|_{\mathcal L_{\ab}^{2,s+1-2/q}},\qquad f\in \mathcal L_{\ab}^{2,s+1-2/q}.
$$
\end{prop}

The proof uses a fractional Sobolev inequality due to Wainger \cite{Wainger}.
\begin{lem}[Wainger]\label{waig}
Let $1<r<q<\infty$. Then
\begin{equation*}
\bigg\| \sum_{k\in\mathbb{Z},\,k\neq0}|k|^{-1/r+1/q}\,\widehat{F}(k)\,\e^{itk}\bigg\|_{L^q_t(0,2\pi)}
\lesssim\|F\|_{L^r(0,2\pi)}, \qquad F\in L^r(0,2\pi),
\end{equation*}
where $\widehat{F}(k)$ is the $k$th Fourier coefficient of $F$.
\end{lem}

\begin{proof}[Proof of Proposition \ref{prop:ext}]
We assume that $\alpha+\beta\neq-1$, the opposite case being similar.
Taking into account that $\lambda_{n}^{\ab}$ are non-zero integers, we apply
Lemma \ref{waig} with $r=2$ to get
\begin{equation*}
\big\|\exp(itL_{\ab})f\big\|_{L_t^q(0,2\pi)}\lesssim\bigg\|\sum\limits_{n=0}^{\infty}\,\e^{it\lambda_n^{\ab}}\,\big(\lambda_n^{\ab}\big)^{1/2-1/q}\,a_n^{\ab}(f)\,\ph\bigg\|_{L_t^2(0,2\pi)}.
\end{equation*}
This estimate combined with Proposition \ref{strichartz} yields
\begin{align*}
\big\|\exp(itL_{\ab})f\big\|_{L_{\theta}^p((0,\pi),\,L_t^q(0,2\pi))}
&\lesssim\bigg\|\sum\limits_{n=0}^{\infty}\,
	\big(\lambda_n^{\ab}\big)^{1/2-1/q}\,a_n^{\ab}(f)\,\ph\bigg\|_{\mathcal L_{\ab}^{2,s}}\\
&=\bigg\|\sum\limits_{n=0}^{\infty}\,\big(\lambda_n^{\ab}\big)^{1/2-1/q+s/2}\,
	a_n^{\ab}(f)\,\ph\bigg\|_2=\|f\|_{\mathcal L_{\ab}^{2,s+1-2/q}}.
\end{align*}
The conclusion follows.
\end{proof}

\section{Proof of Theorem \ref{lpequivf}} \label{sec:g-functions}

Theorem \ref{lpequivf} is a direct consequence of $L^p$-boundedness of $\mathfrak{g}_{\ab}^{\gamma}$ and
standard arguments, see e.g.\ \cite[Section 2]{betsq}. The following result will be proved in
Sections \ref{ssec:g}-\ref{ssec:stand} below.
\begin{thm}\label{ogrgff}
Let $\ab>-1$, $p\in E(\ab)$ and $\gamma>0$. Then $\mathfrak{g}_{\ab}^{\gamma}$ is bounded on $L^p(0,\pi)$.
\end{thm}
We also need to know that $\mathfrak{g}_{\ab}^{\gamma}$ is essentially an isometry on $L^2(0,\pi)$,
or rather a polarized variant of this fact; see, for instance, \cite[Proposition 2.1 (ii)]{betsq}.
\begin{prop}\label{isomf}
Let $\ab>-1$ and $\gamma>0$. Then, for $f,g\in L^2(0,\pi)$,
$$
\langle f,g \rangle = \frac{2^{2\gamma}}{\Gamma(2\gamma)} \int_0^{\pi} \big\langle 
	\partial_t^\gamma H_t^{\ab}f(\theta), \partial_t^\gamma H_t^{\ab}g(\theta) 
	\big\rangle_{L^2(t^{2\gamma-1}dt)} \, d\theta
	+ \chi_{\{\alpha+\beta=-1\}}a_0^{\ab}(f)\,\overline{a_0^{\ab}(g)}.
$$
\end{prop}
In particular, taking above $g=f$ we get
\begin{equation} \label{isometryf}
\|f\|_{2}^2 = \frac{2^{2\gamma}}{\Gamma(2\gamma)} \|\mathfrak{g}_{\ab}^{\gamma}\|_2^2
		+ \chi_{\{\alpha+\beta=-1\}} |a_0^{\ab}(f)|^2, \qquad f \in L^2(0,\pi).
\end{equation}

We are now ready to justify Theorem \ref{lpequivf}, assuming that Theorem \ref{ogrgff} holds.
\begin{proof}[Proof of Theorem \ref{lpequivf}]
In view of Theorem \ref{ogrgff} and the estimate $|a_0^{\ab}(f)| \lesssim \|f\|_p$ (the latter is a
simple consequence of H\"older's inequality), we get
$$
\|\mathfrak g_{\ab}^{\gamma}(f)\|_p+\chi_{\{\alpha+\beta=-1\}}|\,a_0^{\ab}(f)|\lesssim\|f\|_p, \qquad
	f \in L^p(0,\pi).
$$

To show the opposite relation, we use Proposition \ref{isomf} to write 
\begin{align*}
\|f\|_p		=\sup_{g\in L^{p'},\, \|g\|_{p'}=1}|\langle f,g \rangle| 
	&	=\sup_{g\in L^{p'},\, \|g\|_{p'}=1}
	\bigg|\frac{2^{2\gamma}}{\Gamma(2\gamma)} \int_0^{\pi} \big\langle 
	\partial_t^\gamma H_t^{\ab}f(\theta), \partial_t^\gamma H_t^{\ab}g(\theta) 
	\big\rangle_{L^2(t^{2\gamma-1}dt)} \, d\theta \\ & \qquad \qquad \qquad \qquad
	+ \chi_{\{\alpha+\beta=-1\}}a_0^{\ab}(f)\,\overline{a_0^{\ab}(g)} \bigg|.
\end{align*}
Applying now the Cauchy-Schwarz inequality to the inner product under the last integral, and then
H\"older's inequality and $L^{p'}$-boundedness of $\mathfrak{g}_{\ab}^{\gamma}$ (Theorem \ref{ogrgff}),
we conclude that
\begin{align*}
\|f\|_p & \lesssim \sup_{g\in L^{p'},\, \|g\|_{p'}=1}
	\bigg(\frac{2^{2\gamma}}{\Gamma(2\gamma)} \big|\big\langle 
	\mathfrak{g}_{\ab}^{\gamma}(f), \mathfrak{g}_{\ab}^{\gamma}(g)
	\big\rangle\big| + \chi_{\{\alpha+\beta=-1\}} |a_0^{\ab}(f)\,{a_0^{\ab}(g)}| \bigg) \\
& \lesssim \|\mathfrak{g}_{\ab}^{\gamma}(f)\|_p + \chi_{\{\alpha+\beta=-1\}} |a_0^{\ab}(f)|,
\end{align*}
uniformly in $f \in L^p(0,\pi)$.
\end{proof}

It remains to prove Theorem \ref{ogrgff}.
\subsection{{Proof of Theorem \ref{ogrgff}}} \label{ssec:g}
As we shall see, $L^p$-boundedness of $\mathfrak{g}_{\ab}^{\gamma}$ follows in a
straightforward manner from power-weighted
$L^p$-boundedness of an analogous fractional $g$-function in the framework of expansions into
Jacobi trigonometric polynomials. Thus we are going to study weighted counterpart of Theorem \ref{ogrgff}
in the above-mentioned setting. Our main tool will be vector-valued Calder\'on-Zygmund operator theory and its
implementation in the Jacobi context established in \cite{NS1,parameters}. We begin with a brief introduction
of the Jacobi trigonometric polynomial setting. For all these and further facts we refer to 
\cite{NS1,NS2,parameters}.

Let $\ab > -1$. 
The normalized Jacobi trigonometric polynomials are given by
$\P_n^{\ab} = \phi_n^{\ab}/\Psi^{\ab}$, $n \ge 0$. The system $\{\P_n^{\ab}:n \ge 0\}$ is an orthonormal
basis in $L^2((0,\pi),d\mu_{\ab})$, where
$$
d\m_{\ab}(\theta) = \Big(\sin\frac{\theta}2\Big)^{2\alpha+1} 
	\Big( \cos\frac{\theta}2\Big)^{2\beta+1}\, d\theta, \qquad \theta \in (0,\pi).
$$
Each $\P_n^{\ab}$ is an eigenfunction of the Jacobi Laplacian
$$
\J_{\ab} = - \frac{d^2}{d\theta^2} - \frac{\alpha-\beta+(\alpha+\beta+1)\cos\theta}{\sin \theta}
	\frac{d}{d\theta} + \Big( \frac{\alpha+\beta+1}{2}\Big)^2,
$$
the corresponding eigenvalue being $\lambda_n^{\ab}$. Thus $\J_{\ab}$ has a natural self-adjoint
extension in this context (denoted by still the same symbol), whose spectral resolution is given
in terms of $\P_n^{\ab}$.

The semigroup of operators $\{\mathcal{H}_t^{\ab}\}_{t \ge 0}$ generated in $L^2(d\mu_{\ab})$ by means
of the square root of $\J_{\ab}$ is called the Jacobi-Poisson semigroup. We have
\begin{equation} \label{JPpoly}
\mathcal H_t^{\ab}f(\theta) = \sum_{n=0}^{\infty} \exp\Big(-t\sqrt{\lambda_{n}^{\ab}}\Big)
	\big\langle f, \P_n^{\ab}\big\rangle_{d\mu_{\ab}} \P_n^{\ab}(\theta),
\end{equation}
the series being convergent not only in $L^2(d\mu_{\ab})$, but also pointwise if $t>0$.
Actually, the last series converges pointwise for any $f \in L^p(wd\mu_{\ab})$, $w \in A_p^{\ab}$,
$1 \le p < \infty$, providing a definition of $\mathcal{H}_t^{\ab}$, $t >0$, on these weighted spaces.
Here and elsewhere $A_p^{\ab}$ stands for the Muckenhoupt class of weights associated with 
the measure $\mu_{\ab}$ in $(0,\pi)$, see e.g.\ \cite[Section 1]{NS1} for the definition. 
Moreover, $\mathcal{H}_t^{\ab}f(\theta)$
is always a smooth function of $(t,\theta) \in (0,\infty)\times (0,\pi)$. 
All this can be verified with the aid of the bounds, see \eqref{bphi} and 
\cite[Section 2]{NS1},
\begin{align}
|\P_n^{\ab}(\theta)| &\lesssim (n+1)^{\alpha+\beta+2}, \qquad \theta \in (0,\pi), 
	\quad n \ge 0, \label{estjac}\\
\big|\big\langle f,\P_n^{\ab}\big\rangle_{d\mu_{\ab}}\big| & \lesssim \|f\|_{L^p(d\m_{\ab})} 
	(n+1)^{\alpha + \beta + 2}, \qquad n \ge 0; \label{growthFJnew}
\end{align}
here $w \in A_p^{\ab}$ and $1 \le p < \infty$.
There is also an integral representation of $\{\mathcal{H}_t^{\ab}\}_{t>0}$, valid on the weighted
$L^p$ spaces appearing above. We have
$$
\mathcal{H}_t^{\ab}f(\theta) = \int_{0}^{\pi} \mathcal{H}_t^{\ab}(\theta,\varphi)f(\varphi)\, d\m_{\ab}(\varphi),
	\qquad \theta \in (0,\pi), \quad t>0,
$$
where
$$ 
\mathcal{H}_t^{\ab}(\theta,\varphi) =\sum_{n=0}^{\infty} \exp\Big(-t\sqrt{\lambda_{n}^{\ab}}\Big)
	\P_n^{\ab}(\theta)\P_n^{\ab}(\varphi)
$$ 
is the Jacobi-Poisson kernel. A useful integral representation of $\mathcal{H}_t^{\ab}(\theta,\varphi)$
was established in \cite[Proposition 4.1]{NS1} for $\ab \ge -1/2$ and in \cite[Proposition 2.3]{parameters}
in the general case. This representation will implicitly play a crucial role in what follows, however we
decided not to invoke it here due to its complexity. 

Given $\gamma > 0$, we define the vertical fractional square function in the present setting by
\begin{equation*}
{g}_{\ab}^{\gamma}(f)(\theta)= 
	\big\| \partial_t^{\gamma}\mathcal{H}_t^{\ab}f(\theta) \big\|_{L^2(t^{2\gamma-1}dt)}.
\end{equation*} 
This definition makes sense pointwise for $f \in L^p(wd\mu_{\ab})$, $w \in A_p^{\ab}$, $1 \le p < \infty$,
as can be verified by combining \eqref{JPpoly} with \eqref{estjac} and \eqref{growthFJnew};
we leave details to the reader.
The following result not only implies Theorem \ref{ogrgff}, but certainly is also of independent interest. In
particular, it enhances \cite[Corollary 2.5]{NS1} and \cite[Corollary 5.2]{parameters}.
\begin{thm}\label{ogrgf}
Let $\ab>-1$ and $\gamma>0$. Then ${g}_{\ab}^{\gamma}$ is bounded on $L^p(w d\mu_{\ab})$, $w \in A_p^{\ab}$,
$1< p < \infty$, and from $L^1(w d\m_{\ab})$ to weak $L^1(w d\m_{\ab})$, $w \in A_1^{\ab}$.
\end{thm}
We give the proof of Theorem \ref{ogrgf} in Sections \ref{ssec:gP}-\ref{ssec:stand} below. 
First, however, let us see how Theorem \ref{ogrgf} allows us to conclude Theorem \ref{ogrgff}.
\begin{proof}[{Proof of Theorem \ref{ogrgff}}]
We argue similarly as in the proof of \cite[Proposition 2.2]{L2}. Observe that
$$
\mathfrak{g}_{\ab}^{\gamma}(f) = \Psi^{\ab}\, g_{\ab}^{\gamma}(\Psi^{-\alpha-1,-\beta-1}f).
$$
Furthermore, since $w_{\ab} := (\Psi^{\ab})^p/\Psi^{2\alpha+1/2,2\beta+1/2} \in A_p^{\ab}$, $p \in E(\ab)$
(see the proof of \cite[Proposition 2.2]{L2}), Theorem \ref{ogrgf} shows that $g_{\ab}^{\gamma}$
is bounded on $L^p(w_{\ab}d\mu_{\ab})$ when $p \in E(\ab)$. Then we get
\begin{align*}
\|\mathfrak{g}_{\ab}^{\gamma}(f)\|_{p}^p & = \int_0^{\pi} 
	\big|g_{\ab}^{\gamma}(\Psi^{-\alpha-1,-\beta-1}f)(\theta)\big|^p w_{\ab}(\theta)\, d\mu_{\ab}(\theta) \\
& \lesssim \int_0^{\pi} \big|f(\theta) \Psi^{-\alpha-1,-\beta-1}(\theta)\big|^p 
	w_{\ab}(\theta)\, d\mu_{\ab}(\theta)\\
& = \|f\|_p^p,
\end{align*}
uniformly in $f \in L^p(0,\pi)$. The conclusion follows.
\end{proof}

\subsection{{Proof of Theorem \ref{ogrgf}}} \label{ssec:gP}
We employ the theory of Calder\'on-Zygmund operators specified to the space of homogeneous type
$((0,\pi),d\mu_{\ab},|\cdot|)$, where $|\cdot|$ stands for the ordinary distance. Let us briefly
recall the related notions; for more details see \cite{NS1,parameters}.

Let $\mathbb{B}$ be a Banach space and let $K(\theta,\varphi)$ be a kernel defined on
$(0,\pi)\times (0,\pi) \backslash \{(\theta,\varphi):\theta=\varphi\}$ and taking values in $\mathbb{B}$.
We say that $K(\theta,\varphi)$ is a standard kernel 
if it satisfies the growth estimate
\begin{equation} \label{gr}
\|K(\theta,\varphi)\|_{\mathbb{B}} \lesssim \frac{1}{\m_{\ab}(B(\theta,|\theta-\varphi|))}
\end{equation}
and the smoothness estimates
\begin{align}
\| K(\theta,\varphi) - K(\theta',\varphi)\|_{\mathbb{B}} 
	& \lesssim \frac{|\theta-\theta'|}{|\theta-\varphi|}\;
	\frac{1}{\m_{\ab}(B(\theta,|\theta-\varphi|))}, \qquad |\theta-\varphi| > 2|\theta-\theta'|,
	\label{sm1} \\
\| K(\theta,\varphi) - K(\theta,\varphi')\|_{\mathbb{B}} & 
	\lesssim \frac{|\varphi-\varphi'|}{|\theta-\varphi|}\;
	\frac{1}{\m_{\ab}(B(\theta,|\theta-\varphi|))}, \qquad |\theta-\varphi| > 2|\varphi-\varphi'|; \label{sm2}
\end{align}
here $B(\theta,r)$ denotes the ball (interval) centered at $\theta$ and of radius $r$.
As it was observed in \cite[Section 4]{parameters}, even when $K(\theta,\varphi)$ is not scalar-valued, 
the difference conditions
\eqref{sm1} and \eqref{sm2} can be replaced by the more convenient gradient condition
\begin{equation} \label{grad}
\|\partial_{\theta} K(\theta,\varphi)\|_{\mathbb{B}} + \|\partial_{\varphi} K(\theta,\varphi)\|_{\mathbb{B}} \lesssim
\frac{1}{|\theta-\varphi| \m_{\ab}(B(\theta,|\theta-\varphi|))}.
\end{equation}
The derivatives here are taken in the weak sense, which means that for any $\texttt{v}\in\mathbb{B}^*$
\begin{equation} \label{pd}
\langle \texttt{v},\partial_{\theta}K(\theta,\varphi)\rangle
=\partial_{\theta}\langle \texttt{v},K(\theta,\varphi)\rangle
\end{equation}
and similarly for $\partial_{\varphi}.$

A linear operator $T$ assigning to each $f \in L^2(d\m_{\ab})$ a measurable $\mathbb{B}$-valued
function $Tf$ on $(0,\pi)$ is said to be a (vector-valued) Calder\'on-Zygmund operator 
associated with $\mathbb{B}$ if
\begin{itemize}
\item[(a)] $T$ is bounded from $L^2(d\m_{\ab})$ to $L^2_{\mathbb{B}}(d\m_{\ab})$, and
\item[(b)] there exists a standard $\mathbb{B}$-valued kernel $K(\theta,\varphi)$ such that
\begin{equation*}
Tf(\theta) = \int_{0}^{\pi} K(\theta,\varphi) f(\varphi) \, d\m_{\ab}(\varphi), \qquad \textrm{a.e.}
	\;\; \theta \notin \support f,
\end{equation*}
for every $f\in L^2(d\m_{\ab})$ with compact support in $(0,\pi)$.
\end{itemize}
When (b) holds,
we say that $T$ is associated with $K$.

Obviously, $g_{\ab}^{\gamma}$ is not linear, but it can be interpreted in a standard way as a linear
operator
$$
G_{\ab}^{\gamma} \colon f \mapsto \big\{ \partial_{t}^{\gamma}\mathcal{H}_t^{\ab}f\big\}_{t > 0}
$$
mapping into $\mathbb{B}$-valued functions, where $\mathbb{B} = L^2(t^{2\gamma-1}dt)$.
The following result together with a general Calder\'on-Zygmund theory and well-known arguments
(see the proof of \cite[Corollary 2.5]{NS1} and also references given there) justifies Theorem \ref{ogrgf}.
\begin{thm} \label{thm:CZ}
Let $\ab > -1$ and $\gamma > 0$. Then $G_{\ab}^{\gamma}$ is a vector-valued Calder\'on-Zygmund operator
in the sense of the space $((0,\pi),d\mu_{\ab},|\cdot|)$, associated with the Banach space
$\mathbb{B} = L^2(t^{2\gamma-1}dt)$.
\end{thm}

The most difficult step in proving Theorem \ref{thm:CZ} is showing that the vector-valued
kernel
$$
\mathcal{G}_{\ab}^{\gamma}(\theta,\varphi) = 
	\big\{ \partial_t^{\gamma} \mathcal{H}_t^{\ab}(\theta,\varphi)\big\}_{t>0}
$$
satisfies the standard estimates. This is the content of the next lemma.
\begin{lem} \label{stand}
Let $\ab > -1$ and $\gamma > 0$. Then $\mathcal{G}_{\ab}^{\gamma}(\theta,\varphi)$ satisfies \eqref{gr}
and \eqref{grad} with $\mathbb{B} = L^2(t^{2\gamma-1}dt)$.
\end{lem}

On the other hand, $L^2$-boundedness of $G_{\ab}^{\gamma}$ follows readily from the same property
of $g_{\ab}^{\gamma}$ (notice that an analogue of \eqref{isometryf} holds for $g_{\ab}^{\gamma}$).
Moreover, the fact that $G_{\ab}^{\gamma}$ is indeed associated with the kernel 
$\mathcal{G}_{\ab}^{\gamma}(\theta,\varphi)$ can be verified with the aid of quite standard arguments,
following for instance the strategy in the proof of \cite[Proposition 2.5]{L1}. The tools needed
to adapt the reasoning are the estimates \eqref{estjac} and \eqref{growthFJnew}, $L^2$-boundedness
of $G_{\ab}^{\gamma}$ and the growth condition \eqref{gr} for the kernel 
$\mathcal{G}_{\ab}^{\gamma}(\theta,\varphi)$. 

Thus Theorem \ref{thm:CZ}, hence also Theorem \ref{ogrgf}, 
will be justified once we prove Lemma \ref{stand}.

\subsection{{Proof of Lemma \ref{stand}}} \label{ssec:stand}
We will make use of the machinery elaborated in \cite{NS1,parameters}. Therefore we need to invoke some
technical results from \cite{parameters} to make the proof of Lemma \ref{stand} essentially self-contained.
However, we try to be as concise as possible and so for any unexplained symbols or notation we refer to
\cite{parameters}. Let 
$$
q(\theta,\varphi,u,v) = 1 - u \sin\frac{\theta}2 \sin\frac{\varphi}2
	- v \cos\frac{\theta}2 \cos\frac{\varphi}2, \qquad \theta,\varphi \in (0,\pi), \quad u,v \in [-1,1].
$$
We will often omit the arguments and write simply $\q$ instead of $q(\theta,\varphi,u,v)$.
Note that $0 \le \q \le 2$ and $\q \gtrsim |\theta-\varphi|^2$.

\begin{lem}[{\cite[Corollary 3.5]{parameters}}]\label{3.5}
Let $M,N \in \mathbb{N}$ and $L \in \{0 , 1\}$ be fixed. 
The following estimates hold uniformly in $t\in (0,1]$ and $\theta,\varphi \in (0,\pi)$.
\begin{itemize}
\item[(i)]
If $\a,\b \ge -1\slash 2$, then 
\begin{align*}
\big| 
	 \partial_\vp^L \partial_\t^N \partial_t^M 
\mathcal{H}_{t}^{\ab}(\t,\vp)  
\big|  
	\lesssim 
\iint
\frac{\pia \, \pib}{(t^2 + \q)^{   \a + \b + 3\slash 2 + (L+N+M)\slash 2  }}.
\end{align*}
\item[(ii)]
If $-1 < \a < -1\slash 2 \le \b$, then 
\begin{align*}
\big| 
	 \partial_\vp^L \partial_\t^N \partial_t^M 
\mathcal{H}_{t}^{\ab}(\t,\vp)  
\big|  
	&\lesssim 
1 + 
\sum_{K=0,1}
\sum_{k=0,1,2} 
\bigg( \st+\svp \bigg)^{Kk}  \\ 
& \qquad \qquad \times
\iint
\frac{\piK \, \pib}{(t^2 + \q)^{   \a + \b + 3\slash 2 + (L+N+M + Kk)\slash 2  }}.
\end{align*}
\item[(iii)]
If $-1 < \b < -1\slash 2 \le \a$, then 
\begin{align*}
\big| 
	\partial_\vp^L \partial_\t^N \partial_t^M 
\mathcal{H}_{t}^{\ab}(\t,\vp) 
\big| 
	&\lesssim 
1 + 
\sum_{R=0,1}
\sum_{r=0,1,2} 
\bigg( \ct + \cvp \bigg)^{Rr} \\
& \qquad \qquad \times
\iint
\frac{\pia \, \piR}{(t^2 + \q)^{   \a + \b + 3\slash 2 + (L+N+M +Rr)\slash 2  }}.
\end{align*}
\item[(iv)] 
If $-1 < \a,\b < -1\slash 2$, then
\begin{align*}
\big| 
	\partial_\vp^L \partial_\t^N \partial_t^M 
\mathcal{H}_{t}^{\ab}(\t,\vp) 
\big| 
	&\lesssim 
1 + 
\sum_{K,R=0,1}
\sum_{k,r=0,1,2} 
\bigg( \st+\svp \bigg)^{Kk} \bigg( \ct + \cvp \bigg)^{Rr} 
\\ 
& \qquad \qquad \times
\iint
\frac{\piK \, \piR}{(t^2 + \q)^{   \a + \b + 3\slash 2 + (L+N+M + Kk +Rr)\slash 2  }}.
\end{align*}
\end{itemize}
\end{lem}

\begin{lem}[{\cite[Lemma 3.8]{parameters}}]\label{3.8}
Assume that $M,N \in \N$ and $L \in \{0,1\}$ are fixed. Given $\ab > -1$, there exists an
$\epsilon = \epsilon(\ab)>0$ such that
\begin{align*}
& \big| \partial_{\vp}^{L} \partial_{\t}^N \partial_t^M \mathcal{H}_{t}^{\ab}(\t,\vp)\big| \\ & \quad \lesssim
e^{- t \left( \left| \frac{\a + \b + 1}{2}  \right| + \epsilon \right)} 
+ \chi_{\{N=L=0, \, \a+\b+1 \ne 0\}} e^{- t  \left| \frac{\a + \b + 1}{2} \right|}
+ \chi_{\{M=N=L=0, \, \a+\b+1=0\}},
\end{align*}
uniformly in $t \ge 1$ and $\t,\vp \in (0,\pi)$. 
\end{lem}

The next lemma gives control of certain expressions in terms of the right-hand sides of the growth and
gradient conditions. Note that the second estimate is an immediate consequence of the first one and the
bound $\q \gtrsim |\theta-\varphi|^2$.
\begin{lem}[{\cite[Lemma 3.1]{parameters}}]\label{3.1}
Let $\ab > -1$. Assume that $\xi_1,\xi_2,\kappa_1,\kappa_2 \ge 0$ are fixed and such that 
$\a+\xi_1+\kappa_1, \, \b+\xi_2+\kappa_2 \ge -1/2$. Then, uniformly in $\t,\vp \in (0,\pi)$, $\t \ne \vp$,  
\begin{align*}
	&\bigg( \st+\svp \bigg)^{2\xi_1} \bigg( \ct + \cvp \bigg)^{2\xi_2} \iint 
\frac{d\Pi_{\a+\xi_1+\kappa_1}(u) \, 
	d\Pi_{\b+\xi_2+\kappa_2}(v) }{q(\t,\vp,u,v)^{\a+\b+\xi_1+\xi_2+3\slash 2}}\\
	& \quad
\lesssim \frac{1}{\mu_{\ab}(B(\t,|\t-\vp|))}, \\
	&\bigg( \st+\svp \bigg)^{2\xi_1} \bigg( \ct + \cvp \bigg)^{2\xi_2} \iint 
\frac{d\Pi_{\a+\xi_1+\kappa_1}(u) \, d\Pi_{\b+\xi_2+\kappa_2}(v) }{q(\t,\vp,u,v)^{\a+\b+\xi_1+\xi_2+2}}\\
	& \quad
\lesssim \frac{1}{ |\t-\vp| \, \mu_{\ab}(B(\t,|\t-\vp|))}.
\end{align*}
\end{lem}

Finally, we will also need an estimate stated in the next lemma, which does not seem to appear elsewhere.
\begin{lem}\label{3.6}
Let $\eta\in\mathbb{R}$, $\xi > -1$ and $\gamma>0$. Then 
\begin{equation*}
\int_0^1\bigg(\int_0^1\frac{t^{2\gamma-1}\,dt}{\big((t+s)^2+\q\big)^{\eta}}\bigg)^{1/2}
	s^{\xi}\,ds\lesssim
	\begin{cases} 
		\q^{-(\eta-\xi-\gamma-1)/2}, & \eta-\xi-\gamma > 1, \\
		\log(4/\q), & \eta-\xi-\gamma \leq 1,
	\end{cases}
\end{equation*}
uniformly in $\q$.
\end{lem}

\begin{proof}
Denote by $I$ the expression we need to estimate.
Splitting the inner integral and using the elementary relation $\sqrt{A+B}\simeq\sqrt{A}+\sqrt{B}$,
$A,B\ge 0$, we get
\begin{align*}
I & \simeq
\int_0^1\bigg(\int_0^s\frac{t^{2\gamma-1}\,dt}{\big((t+s)^2+\q\big)^{\eta}}\bigg)^{1/2}s^{\xi}\,ds
+ \int_0^1\bigg(\int_s^1\frac{t^{2\gamma-1}\,dt}{\big((t+s)^2+\q\big)^{\eta}}\bigg)^{1/2}s^{\xi}\,ds\\
&\equiv I_1+I_2.
\end{align*}
We will treat $I_1$ and $I_2$ separately.

Observe that in the region of integration in $I_1$ we have
$t+s\simeq s$, so 
\begin{align*}
I_1&\simeq\int_0^1\bigg(\int_0^s {t^{2\gamma-1}\,dt}\bigg)^{1/2}\frac{s^{\xi}\,ds}{(s^2+\q)^{\eta/2}}
\simeq \int_0^1\frac{s^{\xi+\gamma}\,ds}{(s^2+\q)^{\eta/2}}=
\q^{-(\eta-\xi-\gamma-1)/2}\int_0^{1/\sqrt{\q}}\frac{v^{\xi+\gamma}\,dv}{(1+v^2)^{\eta/2}},
\end{align*}
where the last equality is obtained by the change of variable $s=\sqrt{\q}v$. Since
\begin{align*}
\int_0^{1/\sqrt{q}}\frac{v^{\xi+\gamma}\,dv}{(1+v^2)^{\eta/2}}\lesssim
\begin{cases} 
1, & \eta-\xi-\gamma > 1, \\
\log(4/\q), & \eta-\xi-\gamma = 1,\\
\q^{(\eta-\xi-\gamma-1)/2}, & \eta-\xi-\gamma < 1,
\end{cases}
\end{align*}
and clearly $1\lesssim\log(4/\q)$, the desired bound for $I_1$ follows.

To deal with $I_2$ we consider two main cases. If $\q\ge 1$, then $(t+s)^2+\q\simeq 1$ and it is easy to
see that $I_2 \lesssim 1$. This is even stronger estimate than needed.
When $\q<1$, we split the integral in a similar manner as in case of $I_1$ and get
\begin{align*}
I_2&\simeq\int_0^{\sqrt{\q}}\bigg(\int_s^{\sqrt{\q}}\frac{t^{2\gamma-1}\,dt}
	{\big((t+s)^2+\q\big)^{\eta}}\bigg)^{1/2}\,s^{\xi}\,ds+
\int_0^{\sqrt{\q}}\bigg(\int_{\sqrt{\q}}^1\frac{t^{2\gamma-1}\,dt}
	{\big((t+s)^2+\q\big)^{\eta}}\bigg)^{1/2}\,s^{\xi}\,ds\\
&\quad+\int_{\sqrt{\q}}^1\bigg(\int_s^1\frac{t^{2\gamma-1}\,dt}
	{\big((t+s)^2+\q\big)^{\eta}}\bigg)^{1/2}\,s^{\xi}\,ds \equiv J_1+J_2+J_3.
\end{align*}
Notice that $s<t<\sqrt{\q}$ in $J_1$, $s<\sqrt{\q}<t$ in $J_2$ and $\sqrt{\q}<s<t$ in $J_3$. 
Consequently, we have
\begin{align*}
J_1 \simeq
\q^{-\eta/2}\int_0^{\sqrt{\q}}
\bigg(\int_s^{\sqrt{\q}}t^{2\gamma-1}\,dt\bigg)^{1/2} s^{\xi}\,ds
\lesssim\q^{-(\eta-\xi-\gamma-1)/2}.
\end{align*}
In case of $J_2$ we can write
\begin{align*}
J_2\simeq\int_0^{\sqrt{\q}}\bigg(\int_{\sqrt{\q}}^1 t^{-2\eta+2\gamma-1}\,dt\bigg)^{1/2}\,s^{\xi}\,ds.
\end{align*}
Then, assuming that $\eta\neq\gamma$, we get 
$$
J_2\lesssim |1-\q^{-\eta+\gamma}|^{1/2}\q^{(\xi+1)/2}
\le \q^{(\xi+1)/2}+\q^{-(\eta-\xi-\gamma-1)/2}\lesssim1+\q^{-(\eta-\xi-\gamma-1)/2},
$$
while for $\eta=\gamma$ we obtain
$$
J_2\lesssim  (-\log\q)^{1/2}\q^{(\xi+1)/2}\lesssim1.
$$
Finally, considering $J_3$, we have
$$
J_3\simeq\int_{\sqrt{\q}}^1\bigg(\int_s^1 t^{-2\eta+2\gamma-1}\,dt\bigg)^{1/2}\,s^{\xi}\,ds.
$$
Assuming first that $\eta\neq\gamma$, we see that
$$
J_3\lesssim 1+\q^{(\xi+1)/2}+\int_{\sqrt{\q}}^1 s^{-\eta+\xi+\gamma}\,ds
	\lesssim 1+\int_{\sqrt{\q}}^1 s^{-\eta+\xi+\gamma}\,ds,
$$
which easily leads to the bound
$$
J_3 \lesssim \q^{-(\eta-\xi-\gamma-1)/2} + \log(4/\q).
$$
In the remaining case $\eta=\gamma$ we have
$$
J_3\simeq (-\log\q)^{1/2} \int_{\sqrt{\q}}^1 s^{\xi}\,ds
\lesssim \log(4/\q).
$$
Combining the above estimates of $J_1, J_2$ and $J_3$ we get
$$
I_2 \lesssim \q^{-(\eta-\xi-\gamma-1)/2} + \log(4/\q).
$$
Since $\log(4/\q)\lesssim\q^{-(\eta-\xi-\gamma-1)/2}$ if $\eta-\xi-\gamma>1$ and
$\q^{-(\eta-\xi-\gamma-1)/2}<\log(4/\q)$ if $\eta-\xi-\gamma\leq1$,
the necessary bound for $I_2$ follows.
\end{proof}

Now we are in a position to prove Lemma \ref{stand}.
\begin{proof}[{Proof of Lemma \ref{stand}}]
Let $m=\lfloor\gamma\rfloor+1$. In view of the estimates from Lemma \ref{3.5}, 
it is natural and convenient to consider separately the four cases: $\ab\ge-1/2$, $-1<\alpha<-1/2\leq\beta$,
$-1<\beta<-1/2\leq\alpha$ and $-1<\ab<-1/2$. 
The treatment of each of them relies on similar arguments, thus we shall present the details only for the
most involved case $-1 < \ab < -1/2$. Analysis in the other cases is left to the reader.

To show the growth condition \eqref{gr} we split the kernel
\begin{align*}
\partial_t^{\gamma}\mathcal{H}_t^{\ab}(\theta, \varphi)&=\frac{1}{\Gamma(m-\gamma)}
	\int_0^{\infty}\chi_{\{t+s < 1\}}
	\frac{\partial^m}{\partial t^m}\mathcal{H}_{t+s}^{\ab}(\theta, \varphi)\,s^{m-\gamma-1}\,ds\\
&\quad +\frac{1}{\Gamma(m-\gamma)}\int_0^{\infty}\chi_{\{t+s \ge 1\}}
	\frac{\partial^m}{\partial t^m}\mathcal{H}_{t+s}^{\ab}(\theta, \varphi)\,s^{m-\gamma-1}\,ds\equiv A_1+A_2.
\end{align*}
We will estimate $A_1$ and $A_2$ separately.

Using Minkowski's integral inequality and then Lemma \ref{3.5} and Lemma \ref{3.6} 
(the latter applied with $\eta=2\alpha+2\beta+3+m+Kk+Rr$ and $\xi=m-\gamma-1$) we get
\begin{align*}
&\|A_1\|_{L^2(t^{2\gamma-1}\,dt)}\\ 
&\lesssim 1 + 
\sum_{K,R=0,1}
\sum_{k,r=0,1,2} 
\bigg( \st+\svp \bigg)^{Kk} \bigg( \ct + \cvp \bigg)^{Rr}\\
&\quad\times\iint\bigg(\int_0^1\bigg(\int_0^{1}
\frac{t^{2\gamma-1}\,dt}{\big((t+s)^2+\q\big)^{2\alpha+2\beta+3+m+Kk+Rr}}\bigg)^{1/2}
	\,s^{m-\gamma-1}\,ds\bigg)\,\piK \, \piR \\
&\lesssim 1 +
\sum_{K,R=0,1}
\sum_{k,r=0,1,2} 
\bigg( \st+\svp \bigg)^{Kk} \bigg( \ct + \cvp \bigg)^{Rr}\\
&\quad\times\iint\bigg[\Big(\frac{1}{\q}\Big)^{\alpha+\beta+3/2+Kk/2+Rr/2}+\log\frac{4}{\q}
\,\bigg]\,{d\Pi_{\alpha, K}(u)\,d\Pi_{\beta, R}(v)}.
\end{align*}
The term $1$ above satisfies the growth bound, because $\mu_{\ab}((0,\pi)) < \infty$.
The desired estimate for the expression that emerges from considering the first term in the last 
double integral follows directly by an application of Lemma \ref{3.1} 
(specified to $\xi_1=Kk\slash 2$, $\kappa_1=-\a-1\slash 2$ if $K=0$ and $\kappa_1=1 - k\slash 2$ if $K=1$,
$\xi_2=Rr\slash 2$, $\kappa_2=-\b-1\slash 2$ if $R=0$ and $\kappa_2=1 - r\slash 2$ if $R=1$). 
To bound the remaining expression, first recall that $\q \gtrsim |\theta - \varphi|^2$ and observe that
$\log(4/\q) \lesssim \log(4/|\theta - \varphi|)$. On the other hand, we have (see \cite[Lemma 4.2]{NS1})
$$
\mu_{\ab}\big(B(\theta,|\theta-\varphi|)\big) \simeq |\theta-\varphi| (\theta+\varphi)^{2\alpha+1}
	(\pi-\theta + \pi - \varphi)^{2\beta+1}, \qquad \theta,\varphi \in (0,\pi),
$$
so there exists an $\epsilon= \epsilon(\ab)>0$ such that
$$
\mu_{\ab}\big(B(\theta,|\theta-\varphi|)\big) \lesssim |\theta-\varphi|^{\epsilon}, \qquad
	\theta,\varphi \in (0,\pi).
$$
Thus $\log(4/\q)$ is controlled by the right-hand side in \eqref{gr} and the conclusion follows by
finiteness (cf.\ \cite[Section 2]{parameters}) of the measures appearing in the last double integral.

Considering $A_2$, notice that Lemma \ref{3.8} implies that there is $\delta=\delta(\alpha,\beta)>0$ 
such that
$$
\chi_{\{t+s\ge1\}}\Big|\frac{\partial^m}{\partial t^m}\mathcal{H}_{t+s}^{\ab}(\theta, \varphi)\Big|
	\lesssim e^{-(t+s)\,\delta}, \qquad \theta,\varphi\in (0,\pi).
$$
Then using Minkowski's integral inequality we get
$$
\|A_2\|_{L^2(t^{2\gamma-1}\,dt)}\lesssim\int_0^{\infty}\bigg(\int_0^{\infty}e^{-2(t+s)\delta}
	t^{2\gamma-1}\,dt\bigg)^{1/2}\,s^{m-\gamma-1}\,ds < \infty,
$$
which implies the desired bound for $A_2$.

Now we turn to proving the gradient estimate. For symmetry reasons, it is enough to consider 
the partial derivative with respect to $\theta$. 
Let us first ensure that the weak derivative $\partial_{\theta}$ of
$\mathcal{G}^\gamma_{\ab}(\theta,\varphi)$ exists in the sense of \eqref{pd} and is equal to
\{$\partial_{\theta}\partial_t^{\gamma}\mathcal{H}_t^{\ab}(\theta, \varphi)\}_{t>0}$. 
It suffices to check that, for each $\theta,\varphi \in (0,\pi)$, $\theta\neq\varphi$,
$\partial_{\theta}\partial_t^{\gamma}\mathcal{H}_t^{\ab}(\theta, \varphi)\in L^2(t^{2\gamma-1}\,dt)$
and
\begin{equation}\label{equal}
\int_0^{\infty}h(t)\,\partial_{\theta}\partial_t^{\gamma}
\mathcal{H}_t^{\ab}(\theta,\varphi)\,t^{2\gamma-1}\,dt=\partial_{\theta}
\int_0^{\infty}h(t)\,\partial_t^{\gamma}\mathcal{H}_t^{\ab}(\theta, \varphi)\,t^{2\gamma-1}\,dt,
\qquad h\in L^2(t^{2\gamma-1}\,dt).
\end{equation}
The first of these facts is justified by the bounds on $B_1$ and $B_2$ obtained below.
To verify \eqref{equal} we use
Fubini's theorem (its application is legitimate, in view of Schwarz' inequality and the bound for 
$\{\partial_{\theta}\partial_t^{\gamma}\mathcal{H}_t^{\ab}(\theta,\varphi)\}_{t>0}$ proved in a moment).
Take $\theta_1, \theta_2\in(0,\pi)$ such that $\varphi\notin [\theta_1, \theta_2]$. Then
\begin{align*}
\int_{\theta_1}^{\theta_2}\int_0^{\infty}h(t)\,\partial_{\theta}\partial_t^{\gamma}
\mathcal{H}_t^{\ab}(\theta,\varphi)\,t^{2\gamma-1}\,dt d\theta&=\int_0^{\infty}h(t)
\,\partial_t^{\gamma}\mathcal{H}_t^{\ab}(\theta_2, \varphi)\,t^{2\gamma-1}\,dt\\
&\quad-\int_0^{\infty}h(t)\,\partial_t^{\gamma}\mathcal{H}_t^{\ab}(\theta_1, \varphi)\,t^{2\gamma-1}\,dt.
\end{align*}
Dividing both sides of the above equality by $\theta_2-\theta_1$ and taking the limit as
$\theta_1\rightarrow\theta_2$ we get \eqref{equal}.

It remains to show that 
$\|\partial_{\theta}\partial_t^{\gamma}\mathcal{H}_t^{\ab}(\theta,\varphi)\|_{L^2(t^{2\gamma-1}dt)}$
is controlled by the right-hand side of \eqref{grad}.
To proceed, we decompose the kernel in the same way as we did when dealing with the growth condition,
\begin{align*}
\partial_{\theta}\partial_t^{\gamma}\mathcal{H}_t^{\ab}(\theta, \varphi)
&=\frac{1}{\Gamma(m-\gamma)}\int_0^{\infty}\chi_{\{t+s<1\}}\partial_{\theta}\frac{\partial^m}
{\partial t^m}\mathcal{H}_{t+s}^{\ab}(\theta, \varphi)\,s^{m-\gamma-1}\,ds\\
&\quad+\frac{1}{\Gamma(m-\gamma)}\int_0^{\infty}\chi_{\{t+s\ge1\}}\partial_{\theta}
\frac{\partial^m}{\partial t^m}\mathcal{H}_{t+s}^{\ab}(\theta, \varphi)\,s^{m-\gamma-1}\,ds\equiv B_1+B_2.
\end{align*}

Using Minkowski's integral inequality together with Lemma \ref{3.5}, and then Lemma \ref{3.6} 
(specified to $\eta=2\alpha+2\beta+4+m+Kk+Rr$ and $\xi=m-\gamma-1$) we obtain
\begin{align*}
&\|B_1\|_{L^2(t^{2\gamma-1}\,dt)}\\
&\lesssim 1+
\sum_{K,R=0,1}
\sum_{k,r=0,1,2} 
\bigg( \st+\svp \bigg)^{Kk} \bigg( \ct + \cvp \bigg)^{Rr}\\
&\quad\times\iint\bigg(\int_0^1\bigg(\int_0^{1}\frac{t^{2\gamma-1}\,dt}
{\big((t+s)^2+\q\big)^{2\alpha+2\beta+4+m+Kk+Rr}}\bigg)^{1/2}\,s^{m-\gamma-1}\,ds\bigg)\,\piK \, \piR \\
&\lesssim 1+
\sum_{K,R=0,1}
\sum_{k,r=0,1,2} 
\bigg( \st+\svp \bigg)^{Kk} \bigg( \ct + \cvp \bigg)^{Rr}\\
&\quad\times\iint\bigg[\Big(\frac{1}{\q}\Big)^{\alpha+\beta+2+Kk/2+Rr/2}
+\log\frac{4}{\q}\bigg]\,{d\Pi_{\alpha, K}(u)\,d\Pi_{\beta, R}(v)}.
\end{align*}
Now the same arguments as in the case of $A_1$ give the desired estimate. 

As for $B_2$, just notice that by Lemma \ref{3.8} there exists $\delta=\delta(\ab)>0$ such that
$$
\chi_{\{t+s\ge1\}}\Big|\partial_{\theta}\frac{\partial^m}{\partial t^m}
\mathcal{H}_{t+s}^{\ab}(\theta, \varphi)\Big|\lesssim e^{-(t+s)\,\delta}, 
\qquad \theta,\varphi\in (0,\pi).
$$
From here the required bound for $B_2$ follows as in the case of $A_2$.
This completes the proof of Lemma \ref{stand}.
\end{proof}


\end{document}